\documentclass[11pt]{article}
\usepackage{}
\usepackage{amsfonts}
\usepackage{amssymb}
\usepackage{amssymb,amsfonts}
\usepackage{epsfig}
\usepackage{CJK}
\usepackage{color}
\usepackage{amsthm}
\usepackage{amsmath}
\usepackage{graphicx}

 \newtheorem{thm}{Theorem}[section]
 
 \newtheorem{lem}[thm]{Lemma}
 \newtheorem{prop}[thm]{Proposition}
 \theoremstyle{definition}
 
 \newtheorem{rem}[thm]{Remark}
 \numberwithin{equation}{section}
 \def\bR{\mathbb{R}}

\newtheorem{Lemma A.1}{Lemma A.1}
\theoremstyle{definition}

\theoremstyle{remark}

\newcommand{\Real}{\mathbb R}

\begin{document}
\title{Sharp one component regularity for Navier-Stokes  }
\author{Bin Han\footnotemark[1]\and Zhen Lei \footnotemark[2]\and Dong Li\footnotemark[3] \and Na Zhao\footnotemark[4]}
\renewcommand{\thefootnote}{\fnsymbol{footnote}}

\footnotetext[1]{ Department of Mathematics, Hangzhou Dianzi University,
Hangzhou, 310018,  China.}

\footnotetext[2]{ School of Mathematical Sciences, Fudan University,
Shanghai,  200433,  China. }

\footnotetext[3]{ Department of Mathematics, University of British Columbia,
1984 Mathematics road, Vancouver,  Canada V6T1Z2 
and Department of Mathematics, The Hong Kong University of Science \& Technology,
Clear Water Bay, Kowloon, Hong Kong}

\footnotetext[4]{ School of Mathematical Sciences, Fudan University,
Shanghai,  200433,  China. }
\date{}

\maketitle

\begin{abstract}
We consider the conditional  regularity of mild solution $v$ to the incompressible Navier-Stokes equations in three dimensions. Let $e \in \mathbb{S}^2$ and $0 < T^\ast < \infty$.
J. Chemin and  P. Zhang \cite{CP} proved the regularity of $v$ on $(0,T^\ast]$ if there exists $p \in (4, 6)$ such that
$$\int_0^{T^\ast}\|v\cdot e\|^p_{\dot{H}^{\frac{1}{2}+\frac{2}{p}}}dt < \infty.$$
J. Chemin, P. Zhang and Z. F. Zhang \cite{CPZ} extended the range of $p$ to $(4, \infty)$.
In this article we settle the case  $p \in [2, 4]$.  Our proof also works for the case
$p \in (4,\infty)$.
\end{abstract}

\maketitle





\section{Introduction}
Consider the Cauchy problem of the three-dimensional incompressible Navier-Stokes equations on $\Real^3$
\begin{equation}\label{e1.1}
\left\{
 \begin{array}{rlll}
 &\partial_tv+v\cdot\nabla v-\Delta v+\nabla P=0,
\ \ &\ x\in\Real^3,\ t>0,\\
  &\hbox{div}\, v=0, \ \ &\ x\in\Real^3,\ t>0,\\
  &v(0,x)=v_0(x),  \ \ &\ x\in \Real^3.
   \end{array}
  \right.
\end{equation}
Here $v:\, [0,\infty)\times \mathbb R^3\to \mathbb R^3$
 represents the velocity field of the fluid flow and $P:\, [0,\infty)\times \mathbb R^3\to
 \mathbb R$ denotes the pressure. The first two terms represent Newton's acceleration law in Eulerian coordinates whilst the term $-\nabla P$ corresponds to the fluid stress. For the dissipation term
 we have set the kinematic viscosity to be $1$ for simplicity. Since universal
 physical laws should be independent
 of the underlying units (dimension), equation \eqref{e1.1} remains invariant under natural scaling transformations.
  If $(v,P)$ is a solution to \eqref{e1.1}, then for any
 $\lambda>0$,
 \begin{align*}
 v_{\lambda}(t,x) = \lambda v(\lambda^2 t, \lambda x), \quad
 P_{\lambda}(t,x)= \lambda^2 P(\lambda^2 t, \lambda x)
\end{align*}
is also a solution corresponding to rescaled initial data $v_{0,\lambda} (x) = \lambda
v_0(\lambda x)$. Such scaling transformation determines the critical space (norm) for Navier-Stokes
and plays a fundamental role in the wellposedness theory.

The existence of global weak solutions to \eqref{e1.1} is known since the famous work of Leray \cite{Le} (see also Hopf \cite{Ho} for the bounded domain case) for initial data $v_0\in L^2(\Real^3)$ with $\mathrm{div}v_0=0$. The uniqueness and global regularity of Leray-Hopf weak solutions is still one of the most challenging open problems. On the other hand, there exist a vast literature on finite time blowup or non-blowup criterions for local strong solutions. For instance,  the  Prodi-Serrin-Ladyzhenskaya criterion says that
if $$\int_0^{T^\ast}\|v(t, \cdot)\|_{L^q}^pdt < \infty,\quad \frac{2}{p} + \frac{3}{q} = 1$$
for some $3 \leq q \leq \infty$, then $v$ is still regular at time $T^\ast < \infty$, based on a series of important works \cite{Pr, Se, La, St, ESS, Gi}.
We point out that the quantity involved is a dimensionless one with respect to the natural scaling of the Navier-Stokes equations.
Later on, many efforts have been made on weakening the above criterion by imposing constraints only on partial components or directional derivatives of velocity field.
See, for instance, \cite{ZP, CT1, CT2, FQ3, NP, NNP} and the references therein.

In a recent work \cite{CP}, J. Chemin and P. Zhang initiated the following program: To prove the regularity of solutions by only imposing the following assumption
$$I_{p}(v\cdot e) \triangleq \int_0^{T^\ast}\|v\cdot e\|^p_{\dot{H}^{\frac{1}{2}+\frac{2}{p}}}dt < \infty.$$
Here $e \in \mathbb{S}^2$ and $2 \leq p < \infty$. The remarkable feature of the quantity $I_p(v\cdot e)$ lies in the fact that it is a \textit{dimensionless} quantity which only involves \textit{one component} of the velocity field. As an important step towards this line of research,
J. Chemin and P. Zhang \cite{CP} succeeded in the case of $4 < p < 6$, which was subsequently extended by J. Chemin, P. Zhang and Z. F. Zhang \cite{CPZ} to $4 < p < \infty$.
In this article, we give a streamlined proof for all $2 \leq p < \infty$. More precisely, we prove the following theorem.


\begin{thm}\label{t1.2}
Let $v_0 \in \dot H^{\frac{1}{2}}$ with $\nabla\cdot v_0 = 0$ and $\Omega_0 = \nabla\times v_0 \in L^{r_0}$ for some $1<r_0<2$. Let $0 < T^\ast < \infty$ and
 $$v \in C([0,T^*); \dot H^{\frac{1}{2}})\cap L^2([0,T^*);\dot{H}^{\frac{3}{2}})$$
 be the unique local mild solution to the three-dimensional Navier-Stokes equations (\ref{e1.1}) with initial data $v_0$.
If $I_{p}(v\cdot e) < \infty$
for some  $p \in [2,\infty)$ and $e \in \mathbb{S}^2$, then $v \in C([0,T^*];
\dot H^{\frac{1}{2}})\cap L^2([0,T^*];\dot{H}^{\frac{3}{2}})$ and must be regular up to time $T^\ast$:
 more precisely
 \begin{align*}
 \max_{0\le t\le T^\ast } (\| v(t) \|_{\dot H^{\frac 12}} +\| \Omega(t) \|_{L^{r_0}} )<\infty,
 \end{align*}
 and
 for any $ 0 < t_0 < T^\ast$,
 \begin{align*}
 \max_{t_0\le t \le T^{\ast}}  (\| v(t) \|_{ \dot H^1} +
 \| \nabla \Omega (t) \|_{L^{r_0}} ) <\infty.
 \end{align*}
\end{thm}

\begin{rem}
By standard smoothing estimates, the solution $v$ enjoys higher regularity:
$v \in \dot H^m$, $\Omega \in W^{m, r_0}$ for any $m\ge 1$ and any $0<t\le T^\ast$.
\end{rem}

\begin{rem}
In order to simplify the presentation we did not try to lower down the regularity requirement
on initial data although this can certainly be optimised by a more refined analysis. We will pursue
this interesting issue elsewhere. In view of the two-dimensional Biot-Savart law it is of some importance that $\Omega \in L^{r_0}$ for some $1<r_0<2$. The bulk of our analysis will focus
on the case $2\le p\le 4$ which was previously open. The case $4<p<\infty$ can also be treated by our analysis and is included
in a later section. It should be noted that in \cite{CPZ} the case $4<p<\infty$ is treated under the
assumption that the initial vorticity $\Omega_0 \in L^{\frac 32} \cap L^2$. By Sobolev
embedding the condition $\Omega_0 \in L^{\frac 32}$ implies that the initial
velocity $v_0 \in \dot H^{\frac 12}$.
In comparison with
\cite{CPZ} our analysis in the regime $4<p<\infty$ offers a slight relaxation since we only
require $v_0 \in \dot H^{\frac 12}$ with $\Omega_0 \in L^{r_0}$ for \emph{some}
$r_0 \in (1,2)$.

\end{rem}

We now give a brief overview of the proof and explain some main steps.
Without loss of generality, we assume $e=(0,0,1)$ throughout this paper and thus the dimensionless quantity $I_{p}(v\cdot e)$ in the above theorem becomes
$$I_{p}(v^3) = \int_0^{T^*}\|v^3\|^p_{\dot{H}^{\frac{1}{2}+\frac{2}{p}}}dt.$$

{\bf{Step 1.}} Reduction to the two-dimensional vorticity $\omega= -\partial_2 v^1 + \partial_1 v^2$.

For given initial data $v_0\in \dot H^{\frac 12}$, thanks to the smoothing estimates, we have
$v(t) \in \dot H^s$ for any $s\ge 1/2$ immediately on the short time interval $(0,\eta_0]$ for
some $\eta_0>0$ sufficiently small. Therefore by a shift of the time origin if necessary we may
assume without loss of generality that $v_0 \in \dot H^{\frac 12} \cap \dot H^1$. By a similar
reasoning we may also assume $\Omega_0 \in W^{4,r_0}$.   As a first step, we
show that (see Proposition \ref{vH1}): for any $T>0$,
\begin{align*}
\max_{0 \leq t \leq T}\|v\|_{\dot H^1} + \|\nabla v\|_{L^2([0, T],  \dot H^1)} \le2 \|v_0\|_{\dot H^1} \cdot e^{\operatorname{const}\cdot M(T)},
\end{align*}
where
\begin{align*}
   M(T) = \int_0^{T}\| \omega\|^p_{\dot H^{-\frac 12+\frac 2p}}dt+\int_0^{T} \|v^3\|^p_{\dot{H}^{\frac12+\frac2p}}dt.
   \end{align*}
 Whilst the controlling quantity $M(T)$ works for the full range $p\in [2,\infty)$, it should be
noted that for $4<p<\infty$, $-\frac 12 +\frac 2p<0$ and the controlling norm for $\omega$ is
a negative Sobolev norm which is not convenient to use (due to low frequencies) in later
computations.  For this reason we
also prove in Proposition \ref{vH1} (see Remark \ref{rem3.3}) that for
$4<p<\infty$ the quantity $M(T)$ can be replaced by
\begin{align*}
\widetilde M(T)= T \cdot (1+\sup_{0\le t\le T} \|\omega (t)\|_{\tilde r})^p +
 \int_0^T \| v^3 \|_{\dot H^{\frac 12 +\frac 2p} }^p dt,
 \end{align*}
where $\tilde r$ can be any number satisfying $\frac 12 <\frac 1{\tilde r} <\frac 23 (1-\frac 1p)$.

  The preceding argument then establishes a sharp non-blowup criterion: $v$ is regular on $[0, T^\ast]$ if
  $M(T^\ast) < \infty$ (resp. $\widetilde M(T^\ast)<\infty$ for
  $4<p<\infty$). In view of the assumption on $I_p$ in  Theorem \ref{t1.2}, it then suffices for us
  to prove $$\widetilde{I}_p(\omega ) \triangleq \int_0^{T^\ast}\|\omega\|^p_{\dot{H}^{-\frac12+\frac2p}}dt < \infty.$$
 For $p\in (4,\infty)$, it suffices to control
 $$ \sup_{0\le t \le T^\ast} \| \omega (t) \|_{\tilde r}$$
 for some $\frac 12 <\frac 1{\tilde r} <\frac 23 (1-\frac 1p)$.
We also note that the propagation of regularity of $\Omega $ in $W^{4,r_0}$ is not a problem thanks
to the control of $\|v \|_{\dot H^{\frac 12} \cap \dot H^1}$ (see Proposition \ref{vH1}).

{\bf{Step 2.}} Anisotropic decomposition of the velocity.

A remarkable idea introduced in Chemin-Zhang in \cite{CP} is to use the decomposition of the velocity field along horizontal and vertical directions and use the two-dimensional vorticity $\omega$ and $v^3$ as governing unknowns.  Denote
$$\nabla_h=(\partial_1,\partial_2),\ \  \nabla_h^\bot=(-\partial_2,\partial_1) \quad\hbox{and}\quad \Delta_h=\partial_1^2+\partial_2^2.$$
Then,  by using the Biot-Savart's law in the horizontal variables, we have
  \begin{align}\label{vhdecompose}
    &v^h_{\mathrm{curl}} = \nabla_h^\bot\Delta_h^{-1}\omega,\ \ v^h_{\mathrm{div}} = \nabla_h\Delta_h^{-1}\partial_3
    v^3,  \notag \\
    &v^h=v^h_{\mathrm{curl}} - v^h_{\mathrm{div}} \notag \\
     & \;\;\;\;= \begin{pmatrix}
     -\partial_2 \Delta_h^{-1} \omega - \partial_1 \Delta_h^{-1} \partial_3 v^3 \\
     \partial_1 \Delta_h^{-1} \omega - \partial_2 \Delta_h^{-1} \partial_3 v^3
     \end{pmatrix},
  \end{align}
where $$\omega=\partial_1v^2-\partial_2v^1.$$
It is easy to check that\footnote{In  \cite{CP}, Chemin-Zhang considered
$(\omega, \partial_3 v^3)$ as the governing unknowns which is very natural in view of the
physical picture that $v^3$ should be slowly changing in the vertical direction. In order to
control horizontal derivatives Chemin-Zhang used anisotropic spaces carrying positive and negative fractional derivatives in horizontal and vertical directions respectively. In this paper we
found it more convenient to work with the full gradient $\nabla v^3$ in order to trade off
fractional derivatives in the vertical direction.
}
\begin{align}
  &\partial_t \omega+v\cdot\nabla \omega-\Delta \omega=\partial_3v^3\omega+\partial_2v^3\partial_3v^1-\partial_1v^3\partial_3v^2; \label{e_omega_1} \\
  &\partial_t\partial_k v^3 +v\cdot\nabla \partial_k v^3-\Delta \partial_k v^3
  =-\partial_k v\cdot\nabla v^3+\partial_3\partial_k\Delta^{-1}\bigl(\sum_{i,j=1}^3\partial_jv^i\partial_iv^j\bigr), \quad k=1,2,3. \label{e_v3_2}
\end{align}
Thanks to the Biot-Savart's law, the above system written for $(\omega, v^3)$ is equivalent to
the original Navier-Stokes system for $v=(v^1,v^2,v^3)$.

{\bf{Step 3.}} Estimate of  $(\omega, v^3)$. This is the main part of our analysis.
 For fixed $2\le p<\infty$, we shall
choose $1<r<2$, $r$ sufficiently close to $2$, and work with the norms:
\begin{align*}
\| \omega \|_{L^r(\mathbb R^3)}, \quad \| |\nabla_h|^{-\delta} \nabla v^3 \|_2,
\end{align*}
where
\begin{align*}
\delta =\delta(r) = \frac 3r -\frac 32.
\end{align*}
It is not difficult to check that the above two norms have the same scaling
as $\| v\|_{\dot H^{1-\frac 3 r +\frac 32}} \sim \|v\|_{\dot H^{1-}}$
for\footnote{For any quantity $X$ when there is no ambiguity
 we shall use the notation $X+$ to denote $X+\epsilon$ with
sufficiently small $\epsilon$. The notation $X-$ is similarly defined.}
 $r=2-$.
These norms are certainly well-defined since for fixed $r_0$ (recall the initial vorticity $\Omega_0 \in L^{r_0}$ by assumption)
\begin{align*}
\| \omega \|_r + \| |\nabla_h|^{-\delta} \nabla v \|_2
& \lesssim \| \Omega \|_r + \| |\nabla_h|^{-\delta} \Omega \|_2 \notag \\
& \lesssim \| \Omega \|_{W^{4,r_0}},
\end{align*}
if we take $r$ sufficiently close to $2$.

There are several reasons why we choose the norm $\|\omega\|_{2-}$.  Firstly it is natural
to choose $\|\omega\|_p$ norm for some $p$ since in \eqref{e_omega_1} the convection
term $v\cdot \nabla \omega $ will not enter the estimates due to incompressibility. Secondly
if we compute the time derivative of $\| \omega \|_2^2$, then by using \eqref{e_omega_1},
we need to treat the nonlinear terms such as
\begin{align*}
\int_{\mathbb R^3} \partial_2 v^3 \partial_3 v^1 \omega dx
= -\int_{\mathbb R^3}
\partial_2 v^3 \partial_3  \partial_2\Delta_h^{-1} \omega  \cdot \omega dx
- \int_{\mathbb R^3} \partial_2 v^3 \partial_3 \partial_1 \Delta_h^{-1}   \partial_3 v^3
\cdot \omega dx.
\end{align*}
Note that the term $\partial_3 \partial_2 \Delta_h^{-1} \omega $
scales as $|\nabla_h|^{-1} \partial_3 \omega$ for which two-dimensional $L^{\infty}$ embedding
cannot map back to $L^2$.  For this reason one must resort to $\| \omega \|_r$ for some $r<2$.
By a similar reasoning for $v^3$ some negative regularity is needed in the horizontal direction.
This is the one of the reasons for choosing the governing norm as
$\| | \nabla_h|^{-\delta} \nabla v^3\|_2$.

There are a myriad of technical issues in connection with the aforementioned borderline situations.
To get a glimpse into this,  take for example $p=2$ for which $I_p(v^3)$ becomes
\begin{align*}
I_2(v^3) = \int_0^{T^{\ast}} \| v^3 \|_{\dot H^{\frac 32}}^2 dt.
\end{align*}
When computing the time evolution of
$\| \omega \|_r$-norm, we need to estimate a term such as (see Section $4$ for more details)
\begin{align*}
I_{22}:&=\int\partial_2v^3 \partial_2\Delta_h^{-1}\partial_3\omega\omega|\omega|^{r-2}dx.
\end{align*}
The only control we have on $\omega$ is $\|\omega\|_r$ and $\| \nabla \omega \|_r$ (from the diffusion term).
Therefore by using Sobolev embedding and H\"older it is quite natural to bound the above as
\begin{align*}
|I_{22}| \lesssim \| \partial_2 v^3 \|_{L_{x_1,x_2}^2 L_{x_3}^{\infty}}
\| \nabla \omega \|_r \cdot \| \omega \|_r^{r-1}.
\end{align*}
However, even though the quantity $\| \partial_2 v^3 \|_{L_{x_1,x_2}^2 L_{x_3}^{\infty} }$ scales
the same way as $\|v^3 \|_{\dot H^{\frac 32}}$, it cannot be bounded by it
due to the lack of embedding of $\dot H^{\frac 12}$ into $L^{\infty}$ in 1D.  To get around this
problem we perform a refined Littlewood-Paley decomposition in the vertical direction
and manage to obtain a logarithmic inequality of the form:
\begin{align*}
&\int\partial_2v^3|\nabla_h|^{-1}\partial_3\omega\omega|\omega|^{r-2}dx\\
\lesssim &
      \sqrt{\log\big(10+\||\nabla_h|^{-\delta}\partial v^3\|_{L^2}
     +\|\omega\|_{L^r}\big)}\big(\|v^3\|_{\dot{H}^{\frac{3}{2}}}+1\big)
      \|\nabla \omega\|_{L^{r}}
      \|\omega\|_{L^r}^{r-1} \notag \\
 &\;     +(\||\nabla_h|^{-\delta}\partial^2 v^3\|_{L^2}^{\frac12+\delta+
       \epsilon_1}+\||\nabla_h|^{-\delta}\partial^2 v^3\|_{L^2}^{\frac12+\delta-
       \epsilon_1})\cdot
       \frac{1}{1+\|\omega\|^{100}_{L^r}}\big\||\nabla
       \omega||\omega|^{\frac r2-1}\big\|_{L^2}.
\end{align*}
Such estimates turn out to be crucial for the Gronwall argument to work. There are many other
technical issues which cannot be mentioned in this short introduction. In any case
by a very involved analysis on the time evolution
of these norms and taking advantage of the a priori finiteness of $I_p(v^3)$,  we obtain
uniform control of $\|\omega\|_{r} + \||\nabla_h|^{-\delta} \nabla v^3 \|_2$ on the time interval
$[0,T^\ast]$.

{\bf{Step 4.}} Estimate of $\widetilde I_p(\omega)$ for $2\le p \le 4$.  Thanks to the
estimate of $\| \omega \|_r$ in Step 3, the case $4<p<\infty$ is already proven with the help
of Proposition \ref{vH1} and Remark \ref{rem3.3}. To finish the proof of the main theorem it
remains to estimate $\widetilde I_p(\omega)$ for $2\le p\le 4$. Our strategy is to first take $r$ sufficiently close to $2$ for each fixed $2\le p\le 4$, and then use the finiteness
of the scaling-above-critical quantity $\|\omega \|_{L_t^{\infty} L_x^r}$ together with
$\| \nabla ( |\omega|^{\frac r2} )\|_{L_t^2 L_x^2}$ obtained in Step 3 to bound
the critical (dimension-less) quantity $\widetilde I_p(\omega)$. Such a bound is certainly expected
from a scaling heuristic since both $\|\omega\|_{L_t^{\infty} L_x^{2-}}$
and $\|\nabla ( |\omega|^{1-} )\|_{L_t^2 L_x^2}$ carries almost $\dot H^1$ scaling
of velocity.  This then concludes the proof
of the main theorem.

The rest of this paper is organised as follows. In Section \ref{sec_pre} we set up some notation
and collect a few useful lemmas.  In Section 3 we prove Proposition \ref{vH1} which reduces matters
to the control of the horizontal vorticity $\omega$.
In Section 4 and Section 5, we obtain a priori estimates of $\| \omega \|_r $ and
$\| |\nabla_h|^{-\delta} \nabla v^3 \|_2$ for the case $2\le p\le 4$.  In Section 6 we
explain how to do the case $4<p<\infty$. The final section is devoted to the proof of the main
theorem.

\section{Notation and preliminaries}\label{sec_pre}

Let us first recall some Sobolev type inequalities which are relevant to the $L^2$ estimate for $|f|^{\frac r2}$ and $\nabla |f|^{\frac r2}$. The following Lemma will  often be used without
explicit mentioning.

\begin{lem} \label{lem00}
Let the dimension $n\ge 1$.
Fix $k\in \{1,\cdots,n\}$.  Let $1<r<\infty$.
Suppose $f:\, \mathbb R^n \to \mathbb R$ satisfies $\partial_k f \in C^0$ and $f$, $\partial_{kk} f \in L^r(\mathbb R^n)$.
Then $ \partial_k (|f|^{\frac r2}) \in L^2 (\mathbb R^n)$ and
\begin{align*}
- \int_{\mathbb R^3} \partial_{kk} f \cdot |f|^{r-2} f dx
& = (r-1) \int_{f\ne 0} |\partial_k f|^2 |f|^{r-2} dx \notag \\
& = \frac{4(r-1)}{r^2} \| \partial_k ( |f|^{\frac r 2} ) \|^2_{L^2 (f\ne 0)} \notag \\
&= \frac{4(r-1)} {r^2} \| \partial_k (|f|^{\frac r 2} ) \|_{L^2}^2.
\end{align*}
Furthermore for $1<r\le 2$,
\begin{align} \label{ineq_lem00}
\| \partial_k f \|_r \le \frac 2r \| \partial_k (|f|^{\frac r2}) \|_2 \cdot \| f \|_r^{1-\frac r 2}.
\end{align}
For the first group of equalities we also have the following vector-valued version.
Suppose $g:\, \mathbb R^n \to \mathbb R^n$ satisfies $\partial_k g \in C^0$ and $g$, $\partial_{kk} g \in L^r(\mathbb R^n)$.
Then $ \partial_k (|g|^{\frac r2}) \in L^2 (\mathbb R^n)$ and
\begin{align*}
- \int \partial_{kk} g \cdot |g|^{r-2} g  dx
\ge \frac {4(r-1)} {r^2} \| \partial_k ( |g|^{\frac r 2} ) \|_{L^2}^2.
\end{align*}
\end{lem}

\begin{rem}
Dividing both sides of the first group of equalities by the factor $(r-1)$ and taking a suitable limit $r\to 1$ (under some natural assumptions on $f$),
one can derive the analogue of the above for the end-point $r=1$ as
\begin{align*}
- \frac 14 \int \partial_{kk} f \operatorname{sgn}(f)  \log |f| dx =
\| \partial_k ( |f|^{\frac 12} ) \|_2^2.
\end{align*}
For a positive function $f$, this exactly corresponds to the flux (Fisher information) of the entropy functional
$\int  (-f \log f)$.  One should note that in this spirit the entropy is a natural limit
of dissipation law for $|f|^{\frac r2}$  as $r\to 1$.  This gives another explanation why
$-f\log f$ should appear as natural monotone quantities.
\end{rem}

\begin{proof}
It is the regime $1<r<2$ which merits a careful analysis. The first equality follows by a careful integration
by parts (using smooth spatial cut-offs and regularising $|f|$ by $(|f|^2+\epsilon^2)^{\frac 12}$)
and the fact that $\{x: \, f(x)=0, \, \partial_k f(x) \ne 0\}$ has Lebesgue measure zero. The second equality is trivial on the set $f \ne 0$. For the third equality, observe  for $\epsilon \to 0+$,
\begin{align*}
&f_{\epsilon} =
(|f|^2+ \epsilon^2)^{\frac r 4} \to |f|^{\frac r 2}, \quad\text{a.e. in $\mathbb R^n$},\\
&\partial_k f_{\epsilon} = \frac r 2 (|f|^2+\epsilon^2)^{\frac r4-1} f \partial_k f, \\
& \| \partial_k f_{\epsilon} \|_2 \le  \frac r2\| \partial_k f \cdot |f|^{\frac r2-1}
\|_{L^2(f\ne 0)} = \| \partial_k (|f|^{\frac r2}) \|_{L^2(f\ne 0)}.
\end{align*}
It follows easily that $\partial_k (|f|^{\frac r2}) \in L^2$, and
\begin{align*}
\| \partial_k ( |f|^{\frac r2} ) \|_2 \le \| \partial_k (|f|^{\frac r2} ) \|_{L^2(f\ne 0)}.
\end{align*}
Hence the equality holds.

For the  inequality \eqref{ineq_lem00}, one recalls that the set $\{x: \, f(x)=0, \, \partial_k f(x) \ne 0\}$ has
Lebesgue measure zero, and hence
\begin{align*}
\int |\partial_k f|^r &= \int_{f\ne 0} |\partial_k f|^r \cdot |f|^{\frac{r(r-2)}2} \cdot |f |^{\frac{r(2-r)}2} \notag \\
& \le ( \int_{f\ne 0} |\partial_k f |^2 |f|^{r-2} )^{\frac r2} \cdot (\int|f|^r dx)^{\frac{2-r}2}.
\end{align*}

For the last inequality (WLOG again assume $1<r<2$), one first notes that $\partial_k ( |g^j|^{\frac r 2} ) \in L^2$ for each
component $g^j$. Thus $\partial_k( |g|^{\frac r2}) =
\partial_k ( (\sum_{j=1}^n (|g^j|^{\frac r2})^{\frac 4r})^{\frac r 4} ) \in L^2$ by
using the chain rule.  The desired inequality then follows by an argument similar to the
scalar case. We omit details.
\end{proof}

For any $1\le p\le \infty$ and measurable $f:\mathbb R^n \to \mathbb R$,
we will use $\|f\|_{L^p(\mathbb R^n)}$, $\|f\|_{L^p}$ or simply
$\|f\|_p$ to denote the usual $L^p$ norm. For a vector valued function
$f=(f^1,\cdots, f^m)$, we still denote $\|f\|_p :=\sum_{j=1}^m \| f^j \|_p$.

For any $0<T<\infty$ and any Banach space $\mathbb B$ with norm $\| \cdot \|_{\mathbb B}$,
we will use the notation $C([0,T], \,\mathbb B)$ or $C_t^0 \mathbb B$
to denote the space of continuous
$\mathbb B$-valued functions endowed with the norm
\begin{align*}
\| f \|_{C([0,T], \mathbb B)} :=\max_{0\le t \le T} \| f(t) \|_{\mathbb B}.
\end{align*}
Also for $1\le p\le \infty$, we define
\begin{align*}
\| f \|_{L_t^p \mathbb B([0,T])} := \|  \| f(t) \|_{\mathbb B} \|_{L_t^p([0,T])}.
\end{align*}

We shall adopt the following convention for the Fourier transform:
\begin{align*}
& \hat f(\xi) = \int_{\mathbb R^n} f(x) e^{-i x\cdot \xi} dx;\\
& f(x)= \frac 1 {(2\pi)^n} \int_{\mathbb R^n} \hat f (\xi) e^{i x \cdot \xi} d\xi.
\end{align*}
For $s\in \mathbb R$, the fractional Laplacian $|\nabla|^s $ then corresponds to the Fourier
multiplier $|\xi|^s$ defined as
\begin{align*}
\widehat{|\nabla|^s f}(\xi) = |\xi|^s \hat f (\xi),
\end{align*}
whenever it is well-defined. For $s\ge 0$, $1\le p<\infty$, we define the semi-norm and norms:
\begin{align*}
&\| f \|_{\dot W^{s,p}} = \| |\nabla|^s f \|_p, \\
& \| f \|_{W^{s,p}} = \| |\nabla|^s f \|_p + \| f \|_p.
\end{align*}
When $p=2$ we denote $\dot H^s =\dot W^{s,2}$ and $H^s= W^{s,2}$ in accordance with
the usual notation.

For any two quantities $X$ and $Y$, we denote $X \lesssim Y$ if
$X \le C Y$ for some constant $C>0$. Similarly $X \gtrsim Y$ if $X
\ge CY$ for some $C>0$. We denote $X \sim Y$ if $X\lesssim Y$ and $Y
\lesssim X$. The dependence of the constant $C$ on
other parameters or constants are usually clear from the context and
we will often suppress  this dependence. We shall denote
$X \lesssim_{Z_1, Z_2,\cdots,Z_k} Y$
if $X \le CY$ and the constant $C$ depends on the quantities $Z_1,\cdots, Z_k$.

For any two quantities $X$ and $Y$, we shall denote $X\ll Y$ if
$X \le c Y$ for some sufficiently small constant $c$. The smallness of the constant $c$ is
usually clear from the context. The notation $X\gg Y$ is similarly defined. Note that
our use of $\ll$ and $\gg$ here is \emph{different} from the usual Vinogradov notation
in number theory or asymptotic analysis.

 We will need to use the Littlewood--Paley (LP) frequency projection
operators. To fix the notation, let $\phi_0$ be a radial function in
$C_c^\infty(\mathbb{R}^n )$ and satisfy
\begin{equation}\nonumber
0 \leq \phi_0 \leq 1,\quad \phi_0(\xi) = 1\ {\text{ for}}\ |\xi| \leq
1,\quad \phi_0(\xi) = 0\ {\text{ for}}\ |\xi| \geq 7/6.
\end{equation}
Let $\phi(\xi):= \phi_0(\xi) - \phi_0(2\xi)$ which is supported in $\frac 12 \le |\xi| \le \frac 76$.
For any $f \in \mathcal S(\mathbb R^n)$, $j \in \mathbb Z$, define
\begin{align*}
 &\widehat{P_{\le j} f} (\xi) = \phi_0(2^{-j} \xi) \hat f(\xi), \\
 &\widehat{P_j f} (\xi) = \phi(2^{-j} \xi) \hat f(\xi), \qquad \xi \in \mathbb R^n.
\end{align*}
We will denote $P_{>j} = I-P_{\le j}$ ($I$ is the identity operator)
and for any $-\infty<a<b<\infty$, denote
$P_{[a,b]}=\sum_{a\leq j\leq b}P_j$.  Sometimes for simplicity of
notation (and when there is no obvious confusion) we will write $f_j = P_j f$, $f_{\le j} = P_{\le j} f$ and
$f_{a\le\cdot\le b} = \sum_{j=a}^b f_j$.
By using the support property of $\phi$, we have $P_j P_{j^{\prime}} =0$ whenever $|j-j^{\prime}|>1$.

Sometimes it is convenient to use ``fattened" Littlewood-Paley projection operators $\widetilde P_j$
and $\widetilde P_{\ll j}$ defined by
\begin{align*}
&\widehat{\widetilde P_j f} (\xi) =  \phi_1(2^{-j} \xi) \hat f(\xi), \\
&\widehat{\widetilde P_{\ll j} f}(\xi) = \phi_2(2^{-j} \xi) \hat f(\xi),
\end{align*}
where $\phi_1$, $\phi_2\in C_c^{\infty}$ has support in $\{|\xi| \sim 1\}$ and
$\{|\xi| \ll 1\}$ respectively.  As a model case one can consider
$\operatorname{supp}(\phi_1) \subset \{\frac 12 \le |\xi| \le 2\}$ whereras
$\operatorname{supp}(\phi_2) \subset \{ |\xi| \le \frac 14\}$.  The precise numerology does
not play much role in the following computations and estimates as long as their supports
stay well separated.

In section $5$ we will use the following simple (yet powerful) lemma which gives trilinear
para-product decomposition of product of functions. To simplify the notation we shall
write $\int_{\mathbb R^n } (\cdot) dx $ simply as $\int (\cdot)$.
\begin{lem}[Trilinear paraproduct decomposition] \label{lem_trilinear}
For any $f,g,h\in \mathcal{S}(\mathbb{R}^n)$, we have
\begin{align*}
\int fgh =\sum_j\int \big(&f_j g_{[j-3,j+3]}h_{[j-10,j+5]}+f_j g_{[j-3,j+3]}h_{< j-10}\\
          &+f_j g_{<j-3}h_{[j-2,j+2]}
          +f_{<j-3} g_{j}h_{[j-2,j+2]}\big).
\end{align*}
To simplify the notation, we write the above as
\begin{align*}
\int fgh =\sum_j\int \big(&\widetilde{P}_jf \widetilde{P}_jg \widetilde{P}_jh
          +\widetilde{P}_j f \widetilde{P}_jg \widetilde{P}_{\ll j}h\\
          &+\widetilde{P}_jf \widetilde{P}_{\ll j}g \widetilde{P}_jh+\widetilde{P}_{\ll j}f \widetilde{P}_jg \widetilde{P}_jh\big)
\end{align*}
where $\widetilde{P}_j$ and $\widetilde{P}_{\ll j}$ have frequency localized to $\{|\xi|\sim 2^j\}$ and $\{|\xi|\ll 2^j\}$, respectively.
\end{lem}

\begin{proof}
By frequency localization, we have
\begin{align*}
      \int fgh
&=\sum_j\big(\int f_jg_{[j-3,j+3]}h+\int f_jg_{<j-3}h +\int f_jg_{>j+3}h\big)\\
&=\sum_j \int f_j g_{[j-3,j+3]} h_{[j-10,j+5]}
    +\sum_j \int f_j g_{[j-3,j+3]} h_{<j-10}\\
    &\quad+\sum_j \int f_j g_{<j-3} h_{[j-2,j+2]}+\sum_k \int f_{<k-3} g_k h.
\end{align*}
Writing the last term as
$$\sum_j \int f_{<j-3} g_{j} h_{[j-2,j+2]}$$
then yields the result.
\end{proof}

\section{Reduction to $\omega$}\label{sec_BU}

In this section we establish a non-blowup criterion involving only $v^3$ and the horizontal
vorticity $\omega = -\partial_2 v^1 + \partial_1 v^2$.
\begin{prop}\label{vH1}
  Let $0<T<\infty$ and $v \in C_t^0 \dot H^{\frac 12}([0,T) ) \cap L_t^2 \dot H^{\frac 32}([0,T))
  $ be a local mild solution to system \eqref{e1.1} with $v_0 \in \dot H^{\frac 12}$. Let
   $p\in [2,\infty)$. Assume that
  \begin{equation}\nonumber
   M(T) = \int_0^{T}\| \omega\|^p_{\dot H^{-\frac 12+\frac 2p}}dt+\int_0^{T} \|v^3\|^p_{\dot{H}^{\frac12+\frac2p}}dt < \infty.
  \end{equation}
 Then the local solution $v$ can be continued past $T$ and remains regular on $(0, T+\delta]$  for some $\delta>0$. For any $0<t_0<T$,
 $$\max_{t_0 \leq t \leq T}\|v\|_{\dot H^1} + \|\nabla v\|_{L^2([0, T],  \dot H^1)} \le2 \|v(t_0)\|_{\dot H^1} \cdot e^{\operatorname{const}\cdot M(T)} < \infty.$$
 Furthermore if $\Omega_0= \nabla \times v_0 \in L^{r_0}$ for some $r_0 \in (1,2]$,  then
 $\Omega \in C([0,T], L_x^{r_0})$, and for any $0<t_0<T$,
 \begin{align*}
 \sup_{t_0\le t\le T} ( \| \nabla \Omega(t) \|_{L^{r_0}} + \| \nabla^4 \Omega(t) \|_{L^{r_0}}) <\infty.
 \end{align*}
\end{prop}
\begin{rem}
For $p\in [2,4]$, one can replace $\|\omega\|_{\dot H^{-\frac 12+\frac 2p}}$
by the weaker norm $\| |\nabla_h|^{-\frac 12+\frac 2p} \omega\|_{L^2}$.
\end{rem}
\begin{rem} \label{rem3.3}
For $p\in (4,\infty)$, one can replace the quantity $M(T)$ by
\begin{align*}
\widetilde M(T)= T \cdot (1+\sup_{0\le t\le T} \|\omega (t)\|_r)^p +
 \int_0^T \| v^3 \|_{\dot H^{\frac 12 +\frac 2p} }^p dt,
 \end{align*}
where $r$ satisfies $\frac 12 <\frac 1r <\frac 23 (1-\frac 1p)$.  The implied constants in the Gronwall will also depend on $r$ but we shall suppress this dependence. The proof is a simple modification
of the corresponding argument  for $M(T)$. By examining the estimate
of $K_3$ in the proof below, it is clear that
\begin{align*}
|K_3| & \lesssim \| P_{<1} \mathcal R_2(\omega) \|_{r} \| \widetilde P_{<1} ( \partial v
\cdot \partial v ) \|_{\frac r{r-1}} +
\| |\nabla|^{-\frac 12+\frac 2p}
P_{\ge 1} \mathcal R_2(\omega) \|_2 \cdot \| |\nabla|^{\frac 12-\frac 2p} (\partial v
\cdot \partial v) \|_2 \notag \\
& \lesssim \, \|\omega \|_r \cdot \| \nabla v \|_{2}^2 + \|\omega\|_r
\cdot  \| \Delta v\|_2^{2-\frac 2p}
\cdot \| \nabla v\|_2^{\frac 2p} \notag \\
& \le \frac 1 {100} \|\Delta v \|_2^2 + C \cdot
(\| \omega\|_r+ \|\omega\|_r^p) \cdot \| \nabla v\|_2^2.
\end{align*}
A Gronwall argument then concludes the estimates.
\end{rem}
\begin{proof}
By using smoothing estimates we may assume without loss of generality that $t_0=0$
and $v_0 \in \dot H^{\frac 12} \cap \dot H^1$.  We first control $\| v\|_{\dot H^1}$. Applying the spatial derivative $\nabla$ to the Navier-Stokes equations \eqref{e1.1}, and then
taking the $L^2$ inner product of the resulting equations with $\nabla v$,  we have
\begin{align} \notag  
  \begin{split}
   &\frac12\frac{d}{dt}\|\nabla v\|_{L^2}^2
    +\|\Delta v\|_{L^2}^2
    =-\int_{\bR^3}(\partial_iv\cdot\nabla) v\cdot\partial_i vdx\\
    &=-\int_{\bR^3}\partial_iv^3\partial_3v\cdot\partial_i vdx
    -\int_{\bR^3}\partial_iv^h\partial_h v^3\partial_i v^3dx-\int_{\bR^3}\partial_iv^h\partial_h v^{\tilde{h}}\partial_i v^{\tilde{h}}dx\\
    &=K_1+K_2+K_3.
 \end{split}
\end{align}
Here we used Einstein's convention over repeated indices. We emphasis that throughout this paper, the summation over $i$ is always from 1 to 3, but the summation over $h$ and $\tilde{h}$ are always from 1 to 2.

We first estimate $K_1$ and $K_2$.
Clearly for $2\le p\le 4$ (note that $-\frac 12+\frac 2p \ge 0$):
\begin{align*}
|K_1|+|K_2| & \lesssim \| \nabla v^3 \|_{\frac {3p}{2p-2} } \| \nabla v\|_{\frac {6p} {p+2} }^2
\notag \\
& \lesssim \| \nabla v^3 \|_{\dot H^{-\frac 12 +\frac 2p}}
\cdot \| |\nabla|^{1-\frac 1p} \nabla v \|_2^2 \notag \\
& \lesssim \| v^3 \|_{\dot H^{\frac 12 +\frac 2p}}
\| \nabla v \|_2^{\frac 2p} \cdot \| \Delta v \|_2^{2-\frac 2p} \notag \\
&\le \; \frac1{16}\|\Delta v\|_{L^2}^{2}+C\|v^3\|^p_{\dot{H}^{\frac12+\frac2p}}\|\nabla v\|^2_{L^2}.
\end{align*}
Here and below, $C$ represents a constant whose value may change from line to line. On the other
hand for $4<p<\infty$,  noting that $\frac 12 -\frac 2p >0$, we have
\begin{align*}
|K_1|+|K_2|  &\lesssim \| |\nabla|^{-\frac 12+\frac 2p} \partial v^3 \|_2
\cdot \| |\nabla|^{\frac 12-\frac 2p} ( \partial v \cdot \partial v) \|_2 \notag \\
&\lesssim \| v^3 \|_{\dot H^{\frac 12+\frac 2p}}
\cdot \| |\nabla|^{\frac 12-\frac 2p} \partial v\|_3 \cdot \| \partial v \|_6 \notag \\
& \lesssim \| v^3 \|_{\dot H^{\frac 12+\frac 2p}}
\cdot \| \Delta v\|_2 \cdot \| |\nabla|^{1-\frac 2p} \partial v\|_2 \notag \\
& \lesssim \| v^3 \|_{\dot H^{\frac 12+\frac 2p}} \cdot \| \Delta v\|_2^{2-\frac 2p}
\cdot \| \nabla v\|_2^{\frac 2p}.
\end{align*}

For $K_3$, one observes that
\begin{align*}
\partial_h v^{\tilde h} = \mathcal R_2( \omega) + \mathcal R_2( \partial_3 v^3),
\end{align*}
where $\mathcal R_2$ is a two-dimensional Riesz transform.  These terms can
 be estimated in a similar way as in $K_1$ and $K_2$ by using Sobolev norm $\|\cdot\|_{
 \frac {3p}{2p-2}}$ for $2\le p\le 4$ and fractional operator $|\nabla|^{-\frac 12+\frac  2p}$
 for $4<p<\infty$.

Collecting the estimates, we obtain
\begin{equation*}
  \begin{split}
   \frac{d}{dt}\|\nabla v\|_{L^2}^2
    +\|\Delta v\|_{L^2}^2
    \leq C  \|\omega \|^p_{\dot{H}^{-\frac12+\frac2p}}\|\nabla v\|^2_{L^2}
    +C\|v^3\|^p_{\dot{H}^{\frac12+\frac2p}}\|\nabla v\|^2_{L^2}.
 \end{split}
\end{equation*}
Then, the Gronwall inequality gives, for any $0<T_1<T$,
\begin{equation}\nonumber
  \begin{split}
   &\|\nabla v(T_1)\|_{L^2}^2
    +\int_0^{T_1}\|\Delta v(t)\|_{L^2}^2dt\\
    \leq &\|\nabla v_0\|_{L^2}^2\exp\left
    (C\int_0^T\|\omega \|^p_{\dot{H}^{-\frac12+\frac2p}}dt+C\int_0^T \|v^3\|^p_{\dot{H}^{\frac12+\frac2p}}dt\right).
 \end{split}
\end{equation}
On the other hand, for the $\|v \|_{\dot H^{\frac 12}}$-norm, we have
\begin{align}
\frac 12 \frac d {dt} ( \| v \|_{\dot H^{\frac 12}}^2) + \| \nabla |\nabla|^{\frac 12} v
\|_2^2 & \le \| v \|_6 \| \nabla v \|_3 \| |\nabla| v\|_2 \notag \\
 & \lesssim \| \nabla v \|_2^2 \| v \|_{\dot H^{\frac 32}} \notag \\
 & \le \frac 18 \| \nabla |\nabla|^{\frac 12} v \|_2^2 + C \| \nabla v \|_2^4. \notag
 \end{align}
This then easily yields the control of $\| v \|_{\dot H^{\frac 12}}$. Since we have uniform
estimates on $\| v\|_{\dot H^{\frac 12}} + \|v \|_{\dot H^{1}}$ on the time interval
$[0,T)$, the solution $v$ can be continued past $T$.

Finally we show continuity of $\Omega= \nabla \times v$  in $L^{r_0}$ norm.  First we show $\Omega
\in L_t^{\infty} L_x^{r_0} ([0,T])$. Consider the vorticity
equation:
\begin{align*}
\partial_t \Omega + (v \cdot \nabla) \Omega = \Delta \Omega +( \Omega \cdot \nabla) v.
\end{align*}
Clearly
\begin{align*}
\frac 1 {r_0} \frac d {dt} ( \| \Omega\|_{r_0}^{r_0} )
+\frac {4(r_0-1)} {r_0^2}  \| \nabla ( |\Omega|^{\frac {r_0} 2} ) \|_2^2
& \lesssim  \| \nabla v\|_3 \cdot \| |\Omega|^{r_0} \|_{\frac 32} \notag \\
& \lesssim \| v\|_{\dot H^{\frac 32}}  \| |\Omega|^{\frac{r_0} 2} \|_2
\cdot \| \nabla ( |\Omega|^{\frac {r_0}2 } ) \|_2 \notag \\
& \le C \| v \|_{\dot H^{\frac 32}}^2  \| \Omega\|_{r_0}^{r_0}
+ \frac{r_0-1}{r_0^2}  \| \nabla( |\Omega|^{\frac{r_0}2} ) \|_2^2.
\end{align*}
Since we have shown $\|v \|_{L_t^2 \dot H^{\frac 32} ([0,T])} <\infty$, the Gronwall inequality
then easily yields
\begin{align*}
\| \Omega \|_{L_t^{\infty} L_x^{r_0} ([0,T])} +
\| \nabla( |\Omega|^{\frac {r_0}2}) \|_{L_t^2 L_x^2([0,T])} <\infty.
\end{align*}
Now to show continuity in $L^{r_0}$ norm we shall only check the (right) continuity at $t_0=0$.
The continuity at each positive time $t_0\in (0,  T]$ is easier (and omitted) thanks to the usual smoothing effect. For the continuity at $t_0=0$ we only need
to examine the integrals:
\begin{align*}
\int_0^t e^{(t-s) \Delta}  (\Omega \cdot \nabla v)(s)ds,
\quad \text{and } \int_0^t   e^{(t-s)\Delta}  (v \cdot  \nabla \Omega) (s) ds.
\end{align*}

Consider the first integral. We discuss two cases.

Case 1: $\frac 32 \le r_0<2$. Clearly
\begin{align*}
\| \int_0^t e^{(t-s) \Delta}
(\Omega \cdot \nabla v)(s) ds \|_{r_0}
& \lesssim \int_0^t \| \Omega(s) \|_{2r_0}^2 ds  \lesssim
 \int_0^t \| |\Omega |^{\frac{r_0} 2} \|_4^{\frac 4 {r_0}} ds  \notag \\
& \lesssim
\int_0^t \| |\Omega|^{\frac {r_0} 2} \|_2^{\frac 1 {r_0}}
\cdot \| \nabla( |\Omega|^{\frac {r_0} 2} ) \|_2^{\frac 3 {r_0}} ds.
\end{align*}
Since $3/r_0\le 2$, the above clearly tends to zero as $t\to 0+$.

Case 2: $1<r_0<\frac 32$. We have
\begin{align*}
\| \int_0^t e^{(t-s) \Delta}
(\Omega \cdot \nabla v)(s) ds \|_{r_0}
& \lesssim \int_0^t  \| \Omega(s) \|_{2r_0}^2 ds
\lesssim \int_0^t \| |\nabla|^{\frac 52 -\frac 3 {2r_0}} v(s) \|_2^2 ds.
\end{align*}
Since $1< \frac 52-\frac 3{2r_0} <\frac 32$ and $v \in C_t^0 \dot H^{\frac 12}
\cap L_t^2 \dot H^{\frac 32}$, the last integral above also tends to zero as $t\to 0+$.

Now  we consider the integral $\int_0^t   e^{(t-s)\Delta}  (v \cdot  \nabla \Omega) (s) ds$.
By using the property of the mild solution $v$, namely $\lim_{s \to 0+} s^{\frac 12 }
\| v(s) \|_{L_x^{\infty}} =0$, we have (below we also used $\nabla \cdot v=0$)
\begin{align*}
&\|\int_0^t   e^{(t-s)\Delta}  (v \cdot  \nabla \Omega) (s) ds\|_{r_0} \notag \\
& \lesssim \int_0^{\frac t2} (t-s)^{-\frac 12} s^{-\frac 12}
\cdot s^{\frac 12} \| v(s) \|_{\infty} \| \Omega(s) \|_{r_0} ds
+ \int_{\frac t2}^t  s^{-\frac 12} \cdot s^{\frac 12} \|v(s)\|_{\infty}
\| \nabla \Omega (s) \|_{r_0} ds \notag \\
& \lesssim \|s^{\frac 12} v(s) \|_{L_s^{\infty} L_x^{\infty} ([0,t])}
\cdot (\|\Omega (s) \|_{L_s^{\infty} L_x^{r_0} ([0,t])}
+ \| \nabla \Omega(s) \|_{L_s^2 L_x^{r_0} ([0,t])} ) \to 0,
\end{align*}
as $t$ tends to $0+$.  Note that here in the last step we used $$\| \nabla \Omega\|_{r_0}
\lesssim \| \nabla ( |\Omega |^{\frac {r_0} 2} ) \|_2 \cdot \| \Omega\|_{r_0}^{1-\frac{r_0}2}
\lesssim \| \nabla( |\Omega|^{\frac {r_0} 2})\|_2,$$
 which is $L^2$ integrable in time for $1<r_0\le 2$. This finishes the proof
 of $\Omega \in C([0,T], L_x^{r_0})$.

 Finally we note that the estimate for $\nabla \Omega $ is trivial in view of the smoothing effect.
 We omit details.
\end{proof}


\section{Estimate of $\omega$: case $2\le p\le 4$}
In this section we first give the estimate of the horizontal vorticity $\omega$
for the case $2\le p\le 4$. Recall that  $\omega$ satisfies the following equation
\begin{equation}\label{b1}
\partial_t\omega+(v\cdot\nabla) \omega-\Delta \omega=\partial_3 v^3 \omega+
\partial_2v^3\partial_3v^1-\partial_1v^3\partial_3v^2.
\end{equation}
Taking the $L^2$ inner product of equation \eqref{b1} with $\omega|\omega|^{r-2}$, one
has
\begin{align}\label{b2}
\begin{split}
   &\frac{1}{r}\frac{d}{dt}\|\omega\|_{L^r}^{r}
    +\frac{4(r-1)}{r^2}\|\nabla |\omega|^{\frac r2}\|_{L^2}^2\\
  =&\int\partial_3v^3|\omega|^{r}dx +\int\partial_2v^3\partial_3v^1 \omega|\omega|^{r-2}dx -\int\partial_1v^3\partial_3v^2 \omega|\omega|^{r-2}dx\\
  =:&\;I_1+I_2+I_3.
   \end{split}
\end{align}

Let us first estimate the term $I_1$. According to H\"{o}lder  and interpolation inequalities, we have
\begin{align*}
I_1&=\int\partial_3v^3|\omega|^{r}dx\\
&\leq \|\partial_3v^3\|_{L^{\frac{3p}{2(p - 1)}}}\||\omega|^r\|_{L^{\frac{3p}{p + 2}}}\\
&\lesssim \|\partial_3v^3\|_{\dot{H}^{- \frac{1}{2} + \frac{2}{p}}}
\| |\omega|^{\frac{r}{2}} \|^2_{L^{\frac{6p}{p + 2}}}\\
&\lesssim  \|v^3\|_{\dot{H}^{\frac12 + \frac{2}{p}}} \| |\omega|^{\frac{r}{2}} \|^{\frac{2}{p}}_{L^2}
\| \nabla |\omega|^{\frac r2} \|^{2-\frac{2}{p}}_{L^2}.
\end{align*}

The estimate of $I_3$ is similar to $I_2$ and therefore will be omitted. In what follows,  we will focus on the
estimate of $I_2$. Using the decomposition of $v^h$ which is introduced in the introduction (\ref{vhdecompose}), $I_2$ can be rewritten as
\begin{equation}\nonumber
\begin{split}
 I_2&=\int_{\bR^3}\partial_2v^3\partial_3v^1 \omega|\omega|^{r-2}dx\\
 &=-\int_{\bR^3}\partial_2v^3\partial_1\Delta_h^{-1}\partial^2_3v^3 \omega|\omega|^{r-2}dx
 -\int_{\bR^3}\partial_2v^3\partial_2\Delta_h^{-1}\partial_3\omega \omega|\omega|^{r-2}dx\\
 &=: \,I_{21}+I_{22}.
 \end{split}
\end{equation}
Before continuing the estimates, we collect below some useful notation and conventions.

 \textbf{Notation}:
\begin{itemize}
\item For each fixed $2\le p<\infty$, we shall take $r<2$ sufficiently close to $2$.  The explicit
requirement on $r$ can be worked out but for simplicity we shall often suppress it.
We denote
\begin{align*}
\delta = \delta(r) = \frac 3r -\frac 32 >0.
\end{align*}
\item For a scalar function $f=f(x_1,x_2,x_3)$ and $1\le p,q\le \infty$ we use the mixed norm
 notation
\begin{align*}
\| f \|_{L_v^q L_h^p} :=  \Bigl\|      \| f(x_h, x_3 ) \|_{L^p_{x_h}(\mathbb R^2)}
\Bigr\|_{L^q_{x_3}(\mathbb R)}.
\end{align*}
The notation $\| f\|_{L_h^p L_v^q}$ is similarly defined.

\item We use $\nabla$ or $\partial=(\partial_1,\partial_2,\partial_3)$ to denote the usual gradient
operator. Occasionally we also use $\partial^2$ to denote the whole collection
of second order operators $ (\partial_i \partial_j )_{1\le i,j \le 3}$. By Fourier transform,
it is easy to check that
\begin{align*}
& \| |\nabla|^{1-\delta} f \|_2 \le
\| |\nabla_h|^{-\delta} |\nabla | f \|_2 \sim \| |\nabla_h|^{-\delta} \partial f \|_2
\sim \sum_{j=1}^3 \| |\nabla_h|^{-\delta} \partial_j f \|_2, \\
& \| |\nabla|^{2-\delta} f \|_2 \le
\| |\nabla_h|^{-\delta} \Delta f \|_2 \sim
 \| |\nabla_h|^{-\delta} \partial^2 f \|_2
\sim \sum_{i,j=1}^3 \| |\nabla_h|^{-\delta} \partial_i \partial_j f \|_2.
\end{align*}
We will often use these inequalities without explicit mentioning.

\item In various interpolation inequalities we shall use the letter $\epsilon$ to denote a sufficiently small positive constant whose smallness is clear from the context. Such notation is quite useful
in handling certain end-point situations. For example instead of estimating $\| v^3 \|_{L_h^2 L_v^{\infty}}$ we can estimate the scaling-equivalent quantity
$\| v^3 \|_{L_h^{\frac 2 {1-\epsilon}} L_v^{\frac 1 {\epsilon}}}$. The latter can be easily
controlled by $\|v^3 \|_{\dot H^{\frac 12}}$ thanks to Sobolev embedding.

\item The relation of the parameters $p$, $r$ and $\epsilon$ is as follows.
First we fix $p \in [2,\infty)$. After that we will choose $r<2$ (depending on $p$) sufficiently
close to $2$. After $r$ is chosen, we will choose $\epsilon$ sufficiently small in the interpolation
inequalities to get around borderline situations.

\item
  By a slight abuse of notation, we will sometimes
write operators such as $\partial_1 (-\Delta_h)^{-1}$ or $\partial_2 (-\Delta_h)^{-1}$
 simply as $|\nabla_h|^{-1}$ as all
estimates below will hold the same for both operators.
\end{itemize}

We now continue the estimates.
For $I_{21}$, when $p\in[2,4)$, we take $r$ sufficiently close to $2$ and satisfy
$$\max\{\frac{p}{2},\frac 32\}<r<2.$$
Applying  H\"{o}lder  and Sobolev, one can deduce that
\begin{align*}
I_{21} &=\int\partial_2v^3 \partial_1\Delta_h^{-1}\partial_3^2v^3\omega|\omega|^{r-2}dx\\
&\leq\|\partial_2v^3\|_{L_h^{(\frac{1}{r}-\frac{\delta}{2})^{-1}}L_v^{(\frac12+\frac1r-\frac2p)^{-1}}}
      \||\nabla_h|^{-1}\partial_3^2v^3\|_{L_h^{\frac{2}{\delta}}L_v^{2}}
      \big\|\omega|\omega|^{r-2}\big\|_{L_h^{\frac{r}{r-1}}L_v^{(\frac2p-\frac1r)^{-1}}}\\
&\lesssim \|v^3\|_{\dot{H}^{\frac12 + \frac{2}{p}}}
        \||\nabla_h|^{-\delta}\partial_3^2 v^3\|_{L^2}
       \| |\omega|^{\frac{r}{2}} \|^{\frac{2(r-1)}{r}}_{L_h^{2}L_v^{\frac{2p(r-1)}{2r-p}}}\\
       &\lesssim \|v^3\|_{\dot{H}^{\frac12 + \frac{2}{p}}}
        \||\nabla_h|^{-\delta}\partial_3^2 v^3\|_{L^2}
       \| |\omega|^{\frac{r}{2}} \|^{1-\frac2r+\frac{2}{p}}_{L^2}
       \| \nabla |\omega|^{\frac r2} \|^{1-\frac{2}{p}}_{L^2},
\end{align*}
where we recall $\delta=3(\frac{1}{r}-\frac{1}{2})$.

When $p=4$, we take $r$ sufficiently close to $2$.  Then
\begin{align*}
I_{21} &=\int\partial_2v^3 \partial_1\Delta_h^{-1}\partial_3^2v^3\omega|\omega|^{r-2}dx\\
&\leq\|\partial_2v^3\|_{L_v^2 L_h^2}
      \||\nabla_h|^{-1}\partial_3^2v^3\|_{L_v^{2}L_h^{\frac{2}{\delta}}}
      \|\omega|\omega|^{r-2}\|_{L_v^{\infty} L_h^{\frac{2}{1-\delta}}}\\
&\lesssim \|v^3\|_{\dot{H}^{1}}
        \||\nabla_h|^{-\delta}\partial_3^2 v^3\|_{L^2}
       \| |\omega|^{\frac{r}{2}} \|^{\frac{2(r-1)}{r}}_{L_v^{\infty} L_h^{\frac{2}{1-\delta}\frac{2(r-1)}{r}}}\\
       &\lesssim \|v^3\|_{\dot{H}^{\frac12 + \frac{2}{p}}}
        \||\nabla_h|^{-\delta}\partial_3^2 v^3\|_{L^2}
       \| |\omega|^{\frac{r}{2}} \|^{\frac{3}{2}-\frac2r}_{L^2}
       \| \nabla |\omega|^{\frac r2} \|^{\frac{1}{2}}_{L^2}.
\end{align*}
Here for $\| |\omega|^{\frac{r}{2}} \|_{L_v^{\infty} L_h^{\frac{2}{1-\delta}\frac{2(r-1)}{r}}}$,
we have used interpolation inequalities to get
\begin{align*}
&\big\| \| |\omega|^{\frac{r}{2}} \|_{L_h^{\frac{2}{1-\delta}\frac{2(r-1)}{r}}}\big\|
_{L^{\infty}_v}
\lesssim \big\| \| |\nabla_h|^{\frac{2-r}{4(r-1)}} |\omega|^{\frac{r}{2}} \|_{L_v^{\infty}}\big\|_{L^{2}_h}\\
\lesssim
&\big\| \| |\nabla_h|^{\frac{2-r}{4(r-1)}} |\omega|^{\frac{r}{2}} \|_{L_v^{2}}
^{\frac{\frac12-\frac{2-r}{4(r-1)}}{1-\frac{2-r}{4(r-1)}}}
\| |\nabla_3|^{1-\frac{2-r}{4(r-1)}} |\nabla_h|^{\frac{2-r}{4(r-1)}} |\omega|^{\frac{r}{2}} \|_{L_v^{2}}
^{\frac{\frac12}{1-\frac{2-r}{4(r-1)}}}\big\|_{L^{2}_h}\\
\lesssim
&\| |\nabla_h|^{\frac{2-r}{4(r-1)}} |\omega|^{\frac{r}{2}} \|_{L^{2}}
^{\frac{\frac12-\frac{2-r}{4(r-1)}}{1-\frac{2-r}{4(r-1)}}}
\| \nabla |\omega|^{\frac{r}{2}} \|_{L^{2}}^{\frac{\frac12}{1-\frac{2-r}{4(r-1)}}}\\
\lesssim
&\| |\omega|^{\frac{r}{2}} \|_{L^{2}}^{\frac12-\frac{2-r}{4(r-1)}}
\| \nabla |\omega|^{\frac{r}{2}} \|_{L^{2}}^{\frac{2-r}{4(r-1)}\frac{\frac12-\frac{2-r}{4(r-1)}}{1-\frac{2-r}{4(r-1)}}+\frac{\frac12}{1-\frac{2-r}{4(r-1)}}}.
\end{align*}

Let us turn to the estimate of $I_{22}$. First, we consider the case $p\in(2,4)$ which can be easily dealt with
by anisotropic H\"{o}lder inequality and Sobolev embedding. More precisely,
\begin{align*}
  I_{22}:&=\int\partial_2v^3 \partial_2\Delta_h^{-1}\partial_3\omega\omega|\omega|^{r-2}dx\\
&\leq\|\partial_2v^3\|_{L_h^2 L_v^{\frac{p}{p-2}}}
      \||\nabla_h|^{-1}\partial_3 \omega\|_{L_h^{\frac{3}{\delta}}L_v^{r}}
      \|\omega|\omega|^{r-2}\|_{L_h^{\frac{r}{r-1}} L_v^{(\frac{2}{p}-\frac{1}{r})^{-1}}}\\
&\lesssim \|v^3\|_{\dot{H}^{\frac12 + \frac{2}{p}}}
        \|\partial_3 \omega\|_{L^r}
       \| |\omega|^{\frac{r}{2}} \|^{\frac{2(r-1)}{r}}_{L_h^{2}L_v^{\frac{2p(r-1)}{2r-p}}}\\
       &\lesssim \|v^3\|_{\dot{H}^{\frac12 + \frac{2}{p}}}
       \| |\omega|^{\frac{r}{2}} \|^{\frac{2}{p}}_{L^2}
       \| \nabla |\omega|^{\frac r2} \|^{2-\frac{2}{p}}_{L^2}.
\end{align*}

Next we consider $p=4$.
\begin{align*}
  I_{22}:&=\int\partial_2v^3 \partial_2\Delta_h^{-1}\partial_3\omega\omega|\omega|^{r-2}dx\\
&\leq \big\| \|\partial_2v^3\|_{L_h^2}\| |\nabla_h|^{-1}\partial_3\omega\|_{L_h^{\frac{3}{\delta}}}
\| \omega|\omega|^{r-2} \|_{L_h^{\frac{r}{r-1}}} \big\|_{L_v^1}\\
&\lesssim \big\| \|\partial_2v^3\|_{L_h^2}\| \partial_3\omega\|_{L_h^{r}}
\|  |\omega|^{\frac r2} \|_{L_h^{2}}^{\frac{2(r-1)}{r}} \big\|_{L_v^1}\\
&\lesssim \big\| \|\partial_2v^3\|_{L_h^2}\| \partial_3|\omega|^{\frac r2}\|_{L_h^{2}}
\|  |\omega|^{\frac r2} \|_{L_h^{2}} \big\|_{L_v^1}\\
&\lesssim \|v^3\|_{\dot{H}^{1}} \| \nabla |\omega|^{\frac r2} \|_{L^2}
\|  |\omega|^{\frac r2} \|_{L_h^{2}L_v^{\infty}}.
\end{align*}
Here we remark that in the third inequality above, we have used the fact that for smooth $\omega$,
the set $\{ x:\; \omega=0, \, \partial_3 \omega \ne 0\}$ has Lebesgue measure zero.
Therefore when bounding the term $\partial_3 \omega$ one can up to measure zero
regard it as $\partial_3 \omega \cdot 1_{\omega >0} + \partial_3 \omega \cdot 1_{\omega<0}$
and proceed to use interpolation inequalities involving $|\omega|^{\frac r 2}$ which has no
differentiability issues.

Now since
$$ \||\omega|^{\frac r2}\|_{L_h^{2}L_v^{\infty}}
  \lesssim\||\omega|^{\frac r2}\|^{\frac 12}_{L^2}\|\partial_3|\omega|^{\frac r2}\|^{\frac12}_{L^2},$$
one has
\begin{align*}
  I_{22}
\lesssim \|v^3\|_{\dot{H}^{1}} \| \nabla |\omega|^{\frac r2} \|_{L^2}^{\frac32}
\|  |\omega|^{\frac r2} \|_{L^2}^{\frac12}.
\end{align*}

Finally, we consider $p=2$. In this case, Sobolev embedding is not enough. We need to
apply Littlewood-Paley decomposition in the vertical direction and obtain
\begin{align}
  I_{22}:&=\int\partial_2v^3\partial_2\Delta_h^{-1}\partial_3\omega\omega|\omega|^{r-2}dx\notag\\
= &  \int\partial_2P^z_{[-J_0,J_0]}v^3|\nabla_h|^{-1}\partial_3\omega
     \omega|\omega|^{r-2}dx\label{bz2}\\
  & +\sum_{j>J_0}\int\partial_2P^z_{j}v^3|\nabla_h|^{-1}\partial_3\omega
     \omega|\omega|^{r-2}dx\label{bz3} \\
     & +\sum_{j<-J_0} \int\partial_2P^z_{j}v^3|\nabla_h|^{-1}\partial_3\omega
     \omega|\omega|^{r-2}dx,\label{bz1}
\end{align}
where $P^z_{j}$ denotes the Littlewood-Paley decomposition on the vertical variable, and $J_0$ is a positive number which will be determined later.

Estimate of \eqref{bz2}:
\begin{align*}
&\int\partial_2P^z_{[-J_0,J_0]}v^3|\nabla_h|^{-1}\partial_3\omega\omega|\omega|^{r-2}dx\\
\lesssim & \;\|\partial_2P^z_{[-J_0,J_0]}v^3\|_{L_h^{2}L_v^{\infty}}
      \||\nabla_h|^{-1}\partial_3\omega\|_{L_h^{(\frac1r-\frac12)^{-1}}L_v^{r}}
      \big\|\omega|\omega|^{r-2}\big\|_{L_h^{\frac{r}{r-1}}L_v^{\frac{r}{r-1}}}\\
\lesssim &\; \sqrt{J_0}\|v^3\|_{\dot{H}^{\frac{3}{2}}}
      \|\partial_3\omega\|_{L^{r}}
      \|\omega\|_{L^r}^{r-1}.
\end{align*}

Estimate of \eqref{bz3}:

For \eqref{bz3}, we observe that ($\epsilon_1>0$ is a sufficiently small constant)
\begin{align*}
\|\partial_2P^z_{j}v^3\|_{L_h^{2}L_v^{\infty}}
&\lesssim 2^{-j\epsilon_1}\|\partial_2 |\partial_3|^{\frac12+\epsilon_1}P^z_{j}v^3\|_{L^2}
\lesssim 2^{-j\epsilon_1} \| |\nabla|^{\frac 32 +\epsilon_1} v^3 \|_{L^2} \notag \\
& \lesssim 2^{-j\epsilon_1}\||\nabla|^{1-\delta} v^3\|_{L^2}^{\frac12-\delta-
         \epsilon_1}\||\nabla|^{2-\delta}v^3\|_{L^2}^{\frac12+\delta+\epsilon_1}
         \notag \\
&\lesssim 2^{-j\epsilon_1}\||\nabla_h|^{-\delta}\partial v^3\|_{L^2}^{\frac12-\delta-
         \epsilon_1}\||\nabla_h|^{-\delta}\partial^2 v^3\|_{L^2}^{\frac12+\delta+\epsilon_1}.
\end{align*}
Then
\begin{align*}
&\sum_{j>J_0}\int\partial_2P^z_{j}v^3|\nabla_h|^{-1}\partial_3\omega
     \omega|\omega|^{r-2}dx\\
\lesssim &\; 2^{-J_0\epsilon_1}\||\nabla_h|^{-\delta}\partial v^3\|_{L^2}^{\frac12-\delta-
        \epsilon_1}\||\nabla_h|^{-\delta}\partial^2 v^3\|_{L^2}^{\frac12+\delta+\epsilon_1}
        \|\nabla \omega\|_{L^{r}}
      \|\omega\|_{L^r}^{r-1}.
\end{align*}

Estimate of \eqref{bz1}: Note that
$$\| \partial_2 P^z_j v^3 \|_{L_h^2 L_v^{\infty} }
 \lesssim\; 2^{j \epsilon_1} \cdot \|\partial_2 |\partial_3 |^{\frac 12-\epsilon_1}
P^z_j v^3 \|_{L^2}, \quad \text{for negative $j$.}
$$
By an argument similar to \eqref{bz3},  we have
\begin{align*}
&\sum_{j<-J_0}\int\partial_2P^z_{j}v^3|\nabla_h|^{-1}\partial_3\omega
     \omega|\omega|^{r-2}dx\\
\lesssim &\; 2^{-J_0\epsilon_1}\||\nabla_h|^{-\delta}\partial v^3\|_{L^2}^{\frac12-\delta+
        \epsilon_1}\||\nabla_h|^{-\delta}\partial^2 u^3\|_{L^2}^{\frac12+\delta-\epsilon_1}
        \|\nabla \omega\|_{L^{r}}
      \|\omega\|_{L^r}^{r-1}.
\end{align*}

Choosing suitable $J_0$ then yields
\begin{align*}
&\int\partial_2v^3|\nabla_h|^{-1}\partial_3\omega\omega|\omega|^{r-2}dx\\
\lesssim &
      \sqrt{\log\big(10+\||\nabla_h|^{-\delta}\partial v^3\|_{L^2}
     +\|\omega\|_{L^r}\big)}\big(\|v^3\|_{\dot{H}^{\frac{3}{2}}}+1\big)
      \|\nabla \omega\|_{L^{r}}
      \|\omega\|_{L^r}^{r-1} \notag \\
 &\;     +(\||\nabla_h|^{-\delta}\partial^2 v^3\|_{L^2}^{\frac12+\delta+
       \epsilon_1}+\||\nabla_h|^{-\delta}\partial^2 v^3\|_{L^2}^{\frac12+\delta-
       \epsilon_1})\cdot
       \frac{1}{1+\|\omega\|^{100}_{L^r}}
       \| \nabla ( |\omega|^{\frac r2} ) \|_2
\end{align*}

\section{Estimate of $v^3$: case $2\le p\le 4$}
The equation of $v^3$ is
\begin{equation}\label{c1}
\partial_tv^3+(v\cdot\nabla)v^3-\Delta v^3=-\partial_3 P.
\end{equation}
Applying $\partial_k(k=1,2,3)$ to \eqref{c1}, one has
\begin{equation}\label{c2}
\partial_t\partial_kv^3+(v\cdot\nabla) \partial_kv^3+(\partial_kv\cdot\nabla) v^3-
\Delta \partial_k v^3=-\partial_3 \partial_kP.
\end{equation}
Taking the $L^2$ inner product of equation \eqref{c2} with $|\nabla_h|^{-2\delta}\partial_k
v^3$, one has
\begin{align}\nonumber
\frac12 \frac{d}{dt}\bigg(\sum_{k=1}^3\||\nabla_h|^{-\delta}\partial_k v^3\|^2_{L^2}\bigg)
&+\sum_{k=1}^3\||\nabla_h|^{-\delta}\partial_k \nabla v^3\|^2_{L^2}\\
=\sum_{k=1}^3 \bigg(
&-\int (\partial_kv\cdot\nabla) v^3 \cdot |\nabla_h|^{-2\delta}\partial_k v^3 dx\label{c3}\\
&-\int (v\cdot\nabla) \partial_kv^3 \cdot |\nabla_h|^{-2\delta}\partial_k v^3dx \label{c4}\\
&-\int \partial_3 \partial_k P \cdot |\nabla_h|^{-2\delta}\partial_k v^3dx\bigg)
 \label{c5}.
\end{align}
\subsection{Estimate of \eqref{c3}}
Case 1: $\int \partial_kv^3 \partial_3v^3 \cdot |\nabla_h|^{-2\delta}\partial_k v^3 dx$.
We have for all $2\le p<\infty$:
\begin{align*}
 &\int \partial_kv^3 \partial_3v^3 \cdot |\nabla_h|^{-2\delta}\partial_k v^3 dx\notag \\
 & \lesssim \| \partial v^3 \|_{L_v^{(\frac 14 +\frac {\delta} 2+ \epsilon)^{-1}} L_h^{\frac 2{1-\epsilon} }}^2
\cdot \| |\nabla_h|^{-2\delta} \partial v^3 \|_{L_v^{(\frac 12-\delta-2\epsilon)^{-1}} L_h^{\frac 1 {\epsilon} }} \notag \\
& \lesssim  \|
|\nabla|^{\frac 54 -\frac {\delta} 2} v^3
 \|_2^2 \| |\nabla_h|^{-\delta} \partial^2 v^3 \|_2 \notag \\
& \lesssim
  \| |\nabla|^{\frac 2p+\frac 12 } v^3 \|_2 \cdot \| |\nabla|^{-\delta+2-\frac 2p} v^3 \|_2
  \cdot \| |\nabla_h|^{-\delta} \partial^2 v^3 \|_2
  \notag \\
& \lesssim \| v^3 \|_{\dot H^{\frac 12+\frac 2p}} \cdot \| |\nabla_h|^{-\delta} \partial^2 v^3\|_2^{2-\frac 2p}
\cdot \| |\nabla_h|^{-\delta} \partial v^3 \|_2^{\frac 2p}.
\end{align*}

Case 2: $\int(\partial_kv^h \cdot \nabla_h)v^3 \cdot |\nabla_h|^{-2\delta}\partial_k
v^3 dx$.

Case 2a: $\int(\partial_k|\nabla_h|^{-1}\omega \cdot \nabla_h)v^3 \cdot
|\nabla_h|^{-2\delta}\partial_k v^3 dx$.

Applying Littlewood-Paley decomposition on the horizontal direction (see Lemma
\ref{lem_trilinear}, here $\widetilde P^h_j$ corresponds to projection
in $x_h$-variable only), one has
\begin{align}
 &\int(\partial_k|\nabla_h|^{-1}\omega \cdot \nabla_h)v^3 \cdot |\nabla_h|^{-2\delta}
  \partial_k v^3 dx\notag\\
=&\sum_j\bigg[\int(\partial_k |\nabla_h|^{-1}\widetilde{P}_j^h\omega \cdot
  \nabla_h)\widetilde{P}_{\ll j}^hv^3 \cdot |\nabla_h|^{-2\delta}\partial_k\widetilde{P}
   _j^hv^3dx\label{c3_case2a_2}\\
 &+\int(\partial_k |\nabla_h|^{-1}\widetilde{P}_j^h\omega \cdot
  \nabla_h)\widetilde{P}_{j}^hv^3 \cdot |\nabla_h|^{-2\delta}\partial_k\widetilde{P}
   _{\ll j}^hv^3dx\label{c3_case2a_3}\\
 &+\int(\partial_k |\nabla_h|^{-1}\widetilde{P}_{\ll j}^h\omega \cdot
  \nabla_h)\widetilde{P}_{j}^hv^3 \cdot |\nabla_h|^{-2\delta}\partial_k\widetilde{P} _j^hv^3 dx\label{c3_case2a_4}\\
 &+\int(\partial_k |\nabla_h|^{-1}\widetilde{P}_j^h\omega \cdot
  \nabla_h)\widetilde{P}_j^hv^3 \cdot |\nabla_h|^{-2\delta}\partial_k\widetilde{P}
   _j^hv^3dx\bigg].\label{c3_case2a_1}
\end{align}
Estimate of \eqref{c3_case2a_2}:
 For $2\le p<\infty$, by taking $r$ sufficiently close to $2$,  we have
 \begin{align*}
  &  \sum_j \int(\partial_k |\nabla_h|^{-1}\widetilde{P}_j^h\omega \cdot \nabla_h)
     \widetilde{P}_{\ll j}^h v^3 \cdot |\nabla_h|^{-2\delta}\partial_k\widetilde{P}_j^h
     v^3dx\\
 \lesssim & \| (2^j \partial_k |\nabla_h|^{-1} \widetilde P_j^h \omega)
 \|_{l_j^2 L_h^r L_v^r}
 \| (2^{-j} \nabla_h \widetilde P^h_{\ll j } v^3 )
 \|_{l_j^{\infty} L_h^{(\frac 18+\frac{\delta} 4 - \frac {\epsilon} 2)^{-1} }
 L_v^{\frac 1 {\epsilon}}} \notag \\
 & \cdot\|  (|\nabla_h|^{-2\delta}
  \widetilde P_j^h  \partial_k v^3) \|_{l_j^2 L_h^{(\frac 38 -\frac {7\delta} {12}
 +\frac {\epsilon} 2 )^{-1} } L_v^{(1-\frac 1 r - \epsilon)^{-1} } }  \notag \\
 \lesssim & \|\nabla \omega\|_r \| |\nabla|^{\frac 54 -\frac {\delta} 2} v^3
 \|_2^2 \notag \\
 \lesssim & \|\nabla \omega\|_r
 \cdot \| |\nabla|^{\frac 2p+\frac 12 } v^3 \|_2 \cdot \| |\nabla|^{-\delta+2-\frac 2p} v^3 \|_2 \notag \\
 \lesssim & \|\nabla \omega\|_r
 \cdot \|  v^3 \|_{\dot H^{\frac 12 +\frac 2p}}
 \cdot \| |\nabla_h|^{-\delta} \partial v^3 \|_2^{\frac 2p}
 \cdot  \| |\nabla_h|^{-\delta} \partial^2 v^3 \|_2^{1-\frac 2p}.
  \end{align*}

Estimate of \eqref{c3_case2a_3}:
the estimate is similar to the above (one only need to swap $l_j^{\infty}$
and $l_j^2$ in second and third) and therefore omitted.

Estimate of \eqref{c3_case2a_4}:
Clearly  for $2\le p<\infty$,
\begin{align*}
  &  \sum_j\int(\partial_k |\nabla_h|^{-1}\widetilde{P}_{\ll j}^h\omega \cdot
  \nabla_h)\widetilde{P}_{j}^hv^3 \cdot |\nabla_h|^{-2\delta}\partial_k
  \widetilde{P} _j^hv^3 dx \\
 \lesssim & \| (\partial_k |\nabla_h|^{-1}\widetilde{P}_{\ll j}^h\omega )\|
   _{ l_j^{\infty} L_h^{(\frac1r-\frac12)^{-1}} L_v^r}
   \|  (2^{-j(\frac 76 \delta -\epsilon+\frac 14)}\nabla_h \widetilde{P}_{j}^hv^3
   )\|_{ l_j^2  L_h^{(1-\frac1r)^{-1}}L_v^{\frac{1}{\epsilon}}} \\
  &   \cdot \| (2^{j(\frac 76 \delta -\epsilon+\frac 14)} |\nabla_h|^{-2\delta}\widetilde{P}_j^h \partial_k v^3 )\|_{l_j^2 L_h^{2} L_v^{(1-\frac{1}{r}-\epsilon)^{-1}}} \\
\lesssim & \|\nabla \omega\|_r \| |\nabla|^{\frac 14 -\frac {\delta} 2} \partial v^3
 \|_2^2 \notag \\
 \lesssim & \|\nabla \omega\|_r
 \cdot \|  v^3 \|_{\dot H^{\frac 12 +\frac 2p}}
 \cdot \| |\nabla_h|^{-\delta} \partial v^3 \|_2^{\frac 2p}
 \cdot  \| |\nabla_h|^{-\delta} \partial^2 v^3 \|_2^{1-\frac 2p}.
 \end{align*}

Estimate of \eqref{c3_case2a_1}: this is similar to the above and it is omitted.

Case 2b: $\int(\partial_k|\nabla_h|^{-1}\partial_3v^3 \cdot \nabla_h)v^3
\cdot |\nabla_h|^{-2\delta}\partial_k v^3dx$.
\begin{align}
 &\int(\partial_k|\nabla_h|^{-1}\partial_3v^3 \cdot \nabla_h)v^3 \cdot
   |\nabla_h|^{-2\delta} \partial_k v^3 dx\notag\\
=&\sum_j\bigg[\int(\partial_k |\nabla_h|^{-1}\widetilde{P}_j^h\partial_3v^3 \cdot
  \nabla_h)\widetilde{P}_{\ll j}^hv^3 \cdot |\nabla_h|^{-2\delta}\partial_k\widetilde{P}
   _j^hv^3dx\label{c3_case2b_2}\\
 &+\int(\partial_k |\nabla_h|^{-1}\widetilde{P}_j^h\partial_3v^3 \cdot
  \nabla_h)\widetilde{P}_{j}^hv^3 \cdot |\nabla_h|^{-2\delta}\partial_k\widetilde{P}
   _{\ll j}^hv^3dx\label{c3_case2b_3}\\
 &+\int(\partial_k |\nabla_h|^{-1}\widetilde{P}_{\ll j}^h\partial_3v^3 \cdot
  \nabla_h)\widetilde{P}_{j}^hv^3 \cdot
  |\nabla_h|^{-2\delta}\partial_k\widetilde{P}_j^hv^3 dx\label{c3_case2b_4}\\
 &+\int(\partial_k |\nabla_h|^{-1}\widetilde{P}_j^h\partial_3v^3 \cdot
  \nabla_h)\widetilde{P}_j^hv^3 \cdot |\nabla_h|^{-2\delta}\partial_k\widetilde{P}
   _j^hv^3dx\bigg].\label{c3_case2b_1}
\end{align}
Estimate of \eqref{c3_case2b_2}:

We have for all $2\le p<\infty$:
\begin{align*}
  & \sum_j\int(\partial_k |\nabla_h|^{-1}\widetilde{P}_j^h\partial_3v^3 \cdot \nabla_h)
     \widetilde{P}_{\ll j}^hv^3 \cdot |\nabla_h|^{-2\delta}\partial_k\widetilde{P}_j^hv^3
     dx\\
 \lesssim &\|  ( 2^{j(1-\delta)} \partial_k |\nabla_h|^{-1}\widetilde{P}_j^h\partial_3v^3 )\|
     _{ l_j^2 L_h^2 L_v^2}
     \| (2^{-j} \nabla_h \widetilde{P}_{\ll j}^hv^3 )\|
     _{ l_j^{\infty} L_h^{(\frac 18+\frac{\delta} 4-\frac{\epsilon} 2)^{-1}
     } L_v^{\frac 1 {\epsilon} }} \\
  &    \| (2^{j\delta} |\nabla_h|^{-2\delta}\partial_k\widetilde{P}_j^hv^3)\|
     _{l_j^2 L_h^{(\frac 3 8 -\frac{\delta}4 +\frac{\epsilon} 2 )^{-1} } L_v^{\frac{2}{1-2\epsilon}}}
     \\
 \lesssim & \||\nabla_h|^{-\delta}\partial^2 v^3\|_{L^2}
 \cdot \| |\nabla|^{\frac 14 -\frac{\delta} 2} \partial v^3 \|_2^2 \notag \\
 \lesssim & \||\nabla_h|^{-\delta}\partial^2 v^3\|_{L^2} \||\nabla|^{-\delta}
 |\nabla|^{1-\frac 2p}
 \partial v^3\|_{L^2}
    \|v^3\|_{\dot{H}^{\frac12+\frac 2p}} \notag \\
    \lesssim & \||\nabla_h|^{-\delta}\partial^2 v^3\|_{L^2}^{2-\frac 2p}
    \| |\nabla_h|^{-\delta} \partial v^3 \|_2^{\frac 2p}
    \| v^3 \|_{\dot H^{\frac 12+\frac 2p}}.
\end{align*}
Estimate of \eqref{c3_case2b_3}:
for all $2\le p<\infty$,
\begin{align*}
  & \sum_j\int(\partial_k |\nabla_h|^{-1}\widetilde{P}_j^h\partial_3v^3 \cdot \nabla_h)
     \widetilde{P}_{j}^hv^3 \cdot |\nabla_h|^{-2\delta}\partial_k\widetilde{P}_{\ll j}^hv^3
     dx\\
  \lesssim &\|  ( 2^{j(1-\delta)} \partial_k |\nabla_h|^{-1}\widetilde{P}_j^h\partial_3v^3 )\|
     _{ l_j^2 L_h^2 L_v^2}
    \| ( 2^{j(-\frac 14-\frac 12 \delta+\epsilon)} \nabla_h \tilde P^h_j v^3)
    \|_{l_j^2 L_h^2 L_v^{\frac 1 {\epsilon}} }
    \notag \\
    & \quad \cdot \| (2^{j(-\frac 34 +\frac 32 \delta -\epsilon)}
    |\nabla_h|^{-2\delta} \widetilde P_{\ll j}^h \partial_k v^3 )
    \|_{l_j^{\infty} L_h^{\infty} L_v^{\frac 2 {1-2\epsilon} } } \notag \\
 \lesssim & \||\nabla_h|^{-\delta}\partial^2 v^3\|_{L^2}
 \cdot \| |\nabla|^{\frac 14 -\frac{\delta} 2} \partial v^3 \|_2^2 \notag \\
    \lesssim & \||\nabla_h|^{-\delta}\partial^2 v^3\|_{L^2}^{2-\frac 2p}
    \| |\nabla_h|^{-\delta} \partial v^3 \|_2^{\frac 2p}
    \| v^3 \|_{\dot H^{\frac 12+\frac 2p}}.
\end{align*}

 Estimate of \eqref{c3_case2b_4}: for $2\le p<\infty$, we have
\begin{align*}
  & \sum_j\int(\partial_k |\nabla_h|^{-1}\widetilde{P}_{\ll j}^h\partial_3v^3 \cdot
    \nabla_h)\widetilde{P}_{j}^hv^3 \cdot
    |\nabla_h|^{-2\delta}\partial_k\widetilde{P}_{j}^hv^3dx\\
 \lesssim & \| (\partial_k |\nabla_h|^{-1}\widetilde{P}_{\ll j}^h
   \partial_3v^3 )\|_{ l_j^{\infty} L_h^{\frac{2}{\delta}} L_v^2}
   \cdot \| (2^{j(-\frac 14 -\delta+\epsilon)} \nabla_h \widetilde P_j^h v^3 )
   \|_{l_j^2 L_h^{\frac 4{2-\delta}} L_v^{\frac 1 {\epsilon}}} \notag \\
   & \quad \cdot \| (2^{j(\frac 14 +\delta-\epsilon)}
   |\nabla_h|^{-2\delta} \widetilde P^h_j \partial_k v^3 ) \|_{l_j^2 L_h^{\frac 4{2-\delta}}
   L_v^{\frac 2 {1-2\epsilon}}} \notag \\
   \lesssim & \||\nabla_h|^{-\delta}\partial^2 v^3\|_{L^2}
 \cdot \| |\nabla|^{\frac 14 -\frac{\delta} 2} \partial v^3 \|_2^2 \notag \\
    \lesssim & \||\nabla_h|^{-\delta}\partial^2 v^3\|_{L^2}^{2-\frac 2p}
    \| |\nabla_h|^{-\delta} \partial v^3 \|_2^{\frac 2p}
    \| v^3 \|_{\dot H^{\frac 12+\frac 2p}}.
\end{align*}

Estimate of \eqref{c3_case2b_1}:

 The estimate of this term is similar to the above, thus we omit the details.

\subsection{Estimate of \eqref{c4}}
By using integration by parts, one has
\begin{equation*}
  \int(v\cdot\nabla) \partial_kv^3 \cdot |\nabla_h|^{-2\delta}\partial_k v^3dx
   = -\int v \partial_kv^3 \cdot \nabla |\nabla_h|^{-2\delta}\partial_k v^3dx.
\end{equation*}
Case 1: $\int v^3 \partial_kv^3 \partial_3 |\nabla_h|^{-2\delta}\partial_k v^3dx$.
For all $2\le p<\infty$, we have
\begin{align*}
     &\int v^3 \partial_kv^3 \partial_3 |\nabla_h|^{-2\delta}\partial_k v^3dx\\
\lesssim & \|v^3\|_{L_h^{\frac 2{\delta}} L_v^{(\frac 14-\frac{\delta} 2)^{-1}}}
\cdot \| \partial v^3 \|_{L_h^2 L_v^{(\frac 14+\frac{\delta} 2)^{-1}} }
\cdot \| |\nabla_h|^{-2\delta} \partial^2 v^3 \|_{L_h^{\frac 2{1-\delta}} L_v^2} \notag \\
\lesssim & \| |\nabla|^{\frac 54-\frac{\delta} 2} v^3\|_2^2
\cdot \| |\nabla_h|^{-\delta} \partial^2 v^3 \|_2 \notag \\
 \lesssim & \| |\nabla_h|^{-\delta} \partial v^3 \|_2^{\frac 2p}
\cdot \| v^3 \|_{\dot H^{\frac 12+\frac 2p}} \cdot \| |\nabla_h|^{-\delta} \partial^2 v^3
\|_2^{2-\frac 2p}
\end{align*}

Case 2: $\int v^h \partial_kv^3 \cdot \nabla_h |\nabla_h|^{-2\delta}\partial_k v^3dx$.

Case 2a: $\int |\nabla_h|^{-1}\omega \partial_kv^3 \cdot \nabla_h |\nabla_h|^{-2\delta}\partial_k
v^3dx$.
If $2<p\le 4$, then
\begin{align*}
    & \int|\nabla_h|^{-1}\omega \partial_kv^3 \cdot \nabla_h |\nabla_h|^{-2\delta}\partial_k v^3
    dx\\
\lesssim & \| |\nabla_h|^{-1} \omega\|_{L_h^{\frac 3{\delta}} L_v^{(\frac {\delta} 3 +\frac 2p-\frac 12)^{-1}}}
\cdot \| \partial_k v^3 \|_{L_h^2 L_v^{(1-\frac 2 p)^{-1}}}
 \cdot \| |\nabla_h|^{-2\delta} \nabla_h \partial v^3 \|_{(\frac 12-\frac {\delta} 3)^{-1}} \notag \\
\lesssim & \| |\nabla|^{1-\frac 2p } \omega \|_r
\cdot \| v^3 \|_{\dot H^{\frac 12 +\frac 2p}} \cdot \| |\nabla_h|^{-\delta} \partial^2 v^3 \|_2 \notag \\
\lesssim & \| \omega \|_r^{\frac 2p} \| \nabla \omega\|_r^{1-\frac 2p}
\cdot  \| v^3 \|_{\dot H^{\frac 12 +\frac 2p} } \cdot \| |\nabla_h|^{-\delta} \partial^2 v^3 \|_2.
\end{align*}
On the other hand if $p=2$, then
\begin{align*}
    & \int|\nabla_h|^{-1}\omega \partial_kv^3 \cdot \nabla_h |\nabla_h|^{-2\delta}\partial_k v^3
    dx\\
\leq &\||\nabla_h|^{-1}\omega\|_{L_h^{\frac{3}{\delta}}L_v^{r}}
      \|\partial_kv^3\|_{L_h^{\frac{2}{1-\epsilon}}L_v^{\frac{1}{\epsilon}}}
      \|\nabla_h|\nabla_h|^{-2\delta}\partial_k v^3\|_{L_h^{(\frac{1+\epsilon}{2}-
      \frac{\delta}{3})^{-1}}L_v^{(1-\epsilon-\frac1r)^{-1}}}\\
\lesssim & \|\omega\|_{L^r}
        \|v^3\|_{\dot{H}^{\frac32}}
        \||\nabla_h|^{-\delta}\partial^2 v^3\|_{L^2}.
\end{align*}

Case 2b: $\int |\nabla_h|^{-1}\partial_3v^3 \partial_kv^3 \cdot \nabla_h |\nabla_h|^{-2\delta}
\partial_k v^3$. If $p=2$, then
\begin{align*}
     & \int|\nabla_h|^{-1}\partial_3v^3 \partial_kv^3 \cdot \nabla_h |\nabla_h|^{-2\delta}
      \partial_k v^3dx\\
\leq &\||\nabla_h|^{-1}\partial_3v^3\|_{L_h^{\frac{2}{\delta}}L_v^{2}}
      \|\partial_kv^3\|_{L_h^{\frac{2}{1-\epsilon}}L_v^{\frac{1}{\epsilon}}}
      \|\nabla_h|\nabla_h|^{-2\delta}\partial_k v^3\|_{L_h^{\frac{2}{1+\epsilon-\delta}}
      L_v^{\frac{2}{1-2\epsilon}}}\\
\lesssim & \||\nabla_h|^{-\delta}\partial v^3\|_{L^2}
        \|v^3\|_{\dot{H}^{\frac32}}
        \||\nabla_h|^{-\delta}\partial^2 v^3\|_{L^2}.
\end{align*}
If $2<p\le 4$, then
\begin{align*}
     & \int|\nabla_h|^{-1}\partial_3v^3 \partial_kv^3 \cdot \nabla_h |\nabla_h|^{-2\delta}
      \partial_k v^3dx\\
 \lesssim & \| |\nabla_h |^{-1} \partial_3 v^3 \|_{L_h^{\frac 2{\delta-\epsilon}}
 L_v^{(\epsilon+\frac 2p -\frac 12)^{-1}} }
   \cdot \| \partial v^3 \|_{L_h^2 L_v^{(1-\frac 2p)^{-1}}}
   \cdot \| \nabla_h |\nabla_h|^{-2\delta} \partial v^3 \|_{L_h^{\frac 2 {1+\epsilon-\delta}}
   L_v^{\frac 2 {1-2\epsilon}}} \notag \\
   \lesssim & \| |\nabla_h|^{-\delta} |\nabla|^{1-\frac 2p} \partial v^3 \|_2
   \cdot \| v^3 \|_{\dot H^{\frac 12 +\frac 2p}}
   \cdot \| |\nabla_h|^{-\delta} \partial^2 v^3 \|_2 \notag \\
   \lesssim & \| |\nabla_h|^{-\delta} \partial v^3 \|_2^{\frac 2p}
   \cdot \| v^3 \|_{\dot H^{\frac 12+\frac 2p}} \cdot
   \| |\nabla_h|^{-\delta} \partial^2 v^3 \|_2^{2-\frac 2p}.
   \end{align*}

\subsection{Estimate of \eqref{c5}}
$$\int \partial_3 \partial_k P \cdot |\nabla_h|^{-2\delta}\partial_k v^3dx
=-\int\big(\sum_{l,m=1}^3\partial_lu^m\partial_mu^l
       \big)\cdot \partial_3 \partial_k \Delta^{-1}(|\nabla_h|^{-2\delta}\partial_k v^3)dx.$$
Case 1: $l,m\in\{1,2\}$. First observe that if $p=2$, then (below $\mathcal R:=
\partial_3 \partial_k \Delta^{-1}$)
\begin{align*}
  & \int \partial_lv^m\partial_mv^l
    \cdot \partial_3 \partial_k \Delta^{-1}(|\nabla_h|^{-2\delta}\partial_k v^3)dx\\
\lesssim & \|\partial_lv^m\|_{L_h^{(\frac14+\frac{\epsilon}{4}+\frac{\delta}{2})^{-1}}
        L_v^{\frac{2}{1-\epsilon}}}
        \|\partial_mv^l\|_{L_h^{(\frac14+\frac{\epsilon}{4}+\frac{\delta}{2})^{-1}}
        L_v^{\frac{2}{1-\epsilon}}}
        \||\nabla_h|^{-2\delta}\partial_k  \mathcal R v^3\|_{L_h^{(\frac12-\frac{\epsilon}{2}-
        \delta)^{-1}} L_v^{\frac{1}{\epsilon}}} \\
\lesssim & \big(\|\omega\|^2_{L_h^{(\frac14+\frac{\epsilon}{4}+\frac{\delta}{2})^{-1}}
        L_v^{\frac{2}{1-\epsilon}}}
        +\|\partial_3v^3\|^2_{L_h^{(\frac14+\frac{\epsilon}{4}+\frac{\delta}{2})^{-1}}
        L_v^{\frac{2}{1-\epsilon}}}\big)
        \|v^3\|_{\dot{H}^{\frac32}}\\
\lesssim & \big(\| |\nabla|^{\frac12}\omega\|^2_{L^r}
       +\| |\nabla|^{\frac12}|\nabla_h|^{-\delta}\partial_3v^3\|^2_{L^2}\big)
        \|v^3\|_{\dot{H}^{\frac32}}\\
\lesssim & \|\omega\|_{L^r} \|\nabla \omega\|_{L^r}\|v^3\|_{\dot{H}^{\frac32}}
       +\||\nabla_h|^{-\delta}\partial v^3\|_{L^2}\||\nabla_h|^{-\delta}\partial^2v^3\|_{L^2}
        \|v^3\|_{\dot{H}^{\frac32}}.
\end{align*}
On the other hand if $2<p\le 4$, then
\begin{align*}
  & \int \partial_lv^m\partial_mv^l
    \cdot \partial_3 \partial_k \Delta^{-1}(|\nabla_h|^{-2\delta}\partial_k v^3)dx\\
\lesssim & ( \| \omega\|^2_{ L_h^{\frac 4 {1+2\delta}} L_v^p}+
\| \partial_3 v^3\|^2_{ L_h^{\frac 4 {1+2\delta}} L_v^p}
) \| |\nabla_h|^{-2\delta} \partial \mathcal R v^3 \|_{L_h^{(\frac 12-\delta)^{-1}} L_v^{(1-\frac 2p)^{-1}} } \notag \\
\lesssim &  \| \omega \|_r^{\frac 2p} \| \nabla \omega \|_r^{2-\frac 2p} \| v^3 \|_{\dot H^{\frac 12+\frac 2p}} +
\| |\nabla_h|^{-\delta} \partial v^3 \|_2^{\frac 2p}
\| \nabla_h |^{-\delta} \partial^2 v^3 \|_2^{2-\frac 2p}
\| v^3 \|_{\dot H^{\frac 12+\frac 2p}}.
\end{align*}

Case 2: $l=3$ or $m=3$.

The estimate of this term is similar to \eqref{c3} and therefore omitted.

\section{The case $4<p<\infty$}
We shall adopt the same notation as in previous sections. In the following estimates,
we need to use the homogeneous horizontal Besov norm $\|\cdot \|_{\dot B^{h,s}_{p,q}}$ defined for a
three-variable function $f=f(x_h,x_3)=f(x_1,x_2,x_3)$ as:
\begin{align*}
\| f (\cdot, x_3) \|_{\dot B^{h,s}_{p,q}}
:=  \| (2^{js} \| P^h_j f(\cdot, x_3) \|_{L^p_{x_h}})\|_{l_j^q},
\end{align*}
where $s\in \mathbb R$, $1\le p,q\le \infty$, and $P^h_j$ is the Littlewood-Paley projection
operator in the $x_h$ variable.

\subsection{Estimate of $\omega$}

Estimate of $I_1$:   Denote $ g = |\omega|^{\frac r2}$. Then
\begin{align*}
|I_1  | & \lesssim \| |\nabla|^{-\frac 12 +\frac 2p} \partial_3 v^3 \|_2
\cdot \| |\nabla|^{\frac 12 - \frac 2p} ( g^2 ) \|_2 \notag \\
& \lesssim \| v^3 \|_{\dot H^{\frac 12 +\frac 2p}}
\cdot \| |\nabla|^{\frac 12-\frac 2p} g\|_{(\frac 12 - \frac{ \frac 12+\frac 2p} 3)^{-1}}
\cdot \| g \|_{( \frac { \frac 12+\frac 2p} {3 } )^{-1}} \notag \\
& \lesssim \| v^3 \|_{\dot H^{\frac 12+\frac 2p}} \cdot\| |\nabla | g \|_2
\cdot  \| |\nabla|^{1-\frac 2p } g \|_2 \notag \\
& \lesssim \| v^3 \|_{\dot H^{\frac 12 +\frac 2p}}
\cdot \| |\omega|^{\frac r2} \|_2^{\frac 2p} \| \nabla( |\omega|^{\frac r2}) \|_2^{2-\frac 2p}.
\end{align*}

Estimate of $I_{21}$. We have
\begin{align*}
|I_{21}|
& \lesssim
\Bigl\|  \| |\nabla_h|^{-\frac{\delta}3 +\frac 2p-\frac 12 }\partial_2 v^3 \|_{L_h^{(\frac 1r-\frac {\delta}2 +\frac{\epsilon}2 )^{-1}}}
\cdot \bigl( \| |\nabla_h|^{-1} \partial_3^2 v^3 \|_{L_h^{\frac 2 {\delta}}}
\cdot \| |\nabla_h|^{\frac {\delta}3 +\frac 12 -\frac 2 p}  ( \omega |\omega|^{r-2} ) \|_{L_h^{(1-\frac 1r -\frac{\epsilon} 2)^{-1}} }
\notag \\
& \quad + \| |\nabla_h|^{\frac {\delta}3 - \frac 12 -\frac 2p} \partial_3^2 v^3 \|_{L_h^{(\frac 14+\frac 23\delta-\frac 1p)^{-1}}}
\cdot \| |\omega|^{r-1} \|_{L_h^{ (\frac 14 -\frac 12 \delta+\frac 1p -\frac {\epsilon}2 )^{-1} } }  \bigr)
\Bigr\|_{L_v^1} \notag \\
& \lesssim \Bigl\|  \| |\nabla_h|^{\frac 12 +\frac 2p-\epsilon} v^3 \|_{L_h^2}
\cdot \| |\nabla_h|^{-\delta} \partial_3^2 v^3 \|_{L_h^2}
\cdot ( \| \omega\|^{r-1}_{\dot B^{h, \frac{\frac {\delta}3 +\frac 12 -\frac 2p+\epsilon}{r-1} }_{r,2(r-1)}}
+ \| |\nabla_h|^{\frac{\frac {\delta} 3 +\frac 12 -\frac 2p+\epsilon} {r-1} } \omega \|_{L_h^r}^{r-1} )
\Bigr\|_{L_v^1} \notag \\
& \lesssim \| |\nabla_h|^{\frac 12 +\frac 2p-\epsilon} v^3 \|_{L_v^{(\frac 12-\epsilon)^{-1}} L_h^2}\cdot
\| |\nabla_h|^{-\delta} \partial_3^2 v^3 \|_{L_v^2 L_h^2}
\cdot \| \omega\|^{r-1}_{L_v^{\frac{r-1}{\epsilon}}
 \dot B^{h, \frac{\frac {\delta}3 +\frac 12 -\frac 2p+\epsilon}{r-1} }_{r,2(r-1)} }  \notag \\
 & \lesssim  \| v^3 \|_{\dot H^{\frac 12+\frac 2p}} \cdot \| |\nabla_h|^{-\delta} \partial_3^2 v^3 \|_2
 \cdot  \| |\omega|^{\frac{r}{2}} \|^{1-\frac2r+\frac{2}{p}}_{L^2}
       \| \nabla |\omega|^{\frac r2} \|^{1-\frac{2}{p}}_{L^2}.
\end{align*}

Estimate of $I_{22}$.

\begin{align*}
|I_{22}|
& \lesssim \Bigl\|
    \| |\nabla_h|^{-1+\frac 2p+\epsilon } \partial_2 v^3 \|_{L_h^2}
    \cdot \bigl(
        \| |\nabla_h|^{-1} \partial_3 \omega\|_{L_h^{\frac 3 {\delta}}}
             \cdot \| |\nabla_h|^{1 -\frac 2p-\epsilon} ( \omega |\omega|^{r-2} ) \|_{L_h^{(\frac 12 -\frac {\delta} 3 )^{-1}} }
              \notag \\
& \quad        +     \| |\nabla_h|^{ -\frac 2 p-\epsilon} \partial_3 \omega \|_{L_h^{(\frac 1r  - \frac 1p-\frac{\epsilon}2)^{-1} } }
               \cdot \| |\omega|^{r-1} \|_{L_h^{(\frac 12+\frac 1p-\frac 1r+\frac{\epsilon}2)^{-1} } } \bigr)
               \Bigr\|_{L_v^1} \notag \\
 & \lesssim \Bigl\|
  \| |\nabla_h|^{\frac 2p +\epsilon } v^3 \|_{L_h^2}
  \cdot   \| \partial_3 \omega \|_{L_h^r}
  \cdot \| \omega \|^{r-1}_{\dot B^{h,  \frac{ 1-\frac 2p-\epsilon} {r-1}}_{ r, 2(r-1)}}
  \Bigr\|_{L_v^1} \notag \\
  & \lesssim
   \| |\nabla_h|^{\frac 2p +\epsilon } v^3 \|_{L_v^{\frac 1 {\epsilon}} L_h^2}
  \cdot \| \nabla \omega\|_{L_v^r L_h^r} \cdot \| \omega \|^{r-1}_{L_v^{(1-\frac 1r-\epsilon)^{-1}(r-1)}
  \dot B^{h,  \frac{ 1-\frac 2p-\epsilon} {r-1}}_{ r, 2(r-1)}}
   \notag \\
  & \lesssim \| v^3 \|_{\dot H^{\frac 12+\frac 2p}} \cdot \| \nabla ( |\omega|^{\frac r2} ) \|_{L^2}^{2 -\frac 2p}
  \cdot \| |\omega|^{\frac r2} \|_{L^2}^{\frac 2p}.
 \end{align*}

\subsection{Estimate of $v^3$}

\subsubsection{Estimate of \eqref{c3}}
This is already done for $2\le p<\infty$ in the previous sections.

\subsubsection{Estimate of \eqref{c4}}
Recall by using integration by parts, one has
\begin{equation*}
  \int(v\cdot\nabla) \partial_kv^3 \cdot |\nabla_h|^{-2\delta}\partial_k v^3dx
   = -\int v \partial_kv^3 \cdot \nabla |\nabla_h|^{-2\delta}\partial_k v^3dx.
\end{equation*}

Case 1: $\int v^3 \partial_kv^3 \partial_3 |\nabla_h|^{-2\delta}\partial_k v^3dx$. This is already
done for $2\le p<\infty$.

Case 2: $\int v^h \partial_kv^3 \cdot \nabla_h |\nabla_h|^{-2\delta}\partial_k v^3dx$.

Case 2a: $\int |\nabla_h|^{-1}\omega \partial_kv^3 \cdot \nabla_h |\nabla_h|^{-2\delta}\partial_k
v^3dx$.

Denote $T= \nabla_h |\nabla_h|^{-2\delta}$. Then (in the following
computation we used a commutator estimate which is proved in
\cite{Li16} for more general operators)
\begin{align*}
 & 2 \int |\nabla_h|^{-1}\omega \partial_kv^3 \cdot \nabla_h |\nabla_h|^{-2\delta}\partial_k
v^3 dx\notag \\
= & -\int (T( |\nabla_h|^{-1} \omega \cdot \partial_k v^3) - |\nabla_h|^{-1} \omega T  \partial_k
v^3 ) \partial_k v^3 dx\notag \\
\lesssim & \Bigl\|  \| (T( |\nabla_h|^{-1} \omega \cdot \partial_k v^3) - |\nabla_h|^{-1} \omega T  \partial_k
v^3 )\|_{L_h^{(\frac 58 +\frac 1{2p} -\frac{\delta}4 -\frac{\epsilon} 4)^{-1}}}
\cdot \| \partial_k v^3 \|_{L_h^{(\frac 38 -\frac 1{2p} +\frac{\delta}4+\frac{\epsilon}4)^{-1}}}
\Bigr\|_{L_v^1} \notag \\
\lesssim &
\Bigl\|  \| |\nabla_h|^{-2\delta} \omega \|_{L_h^{(\frac 14 +\frac 1p -\frac{\delta}2-\frac{\epsilon}2)^{-1}}}  \cdot \| \partial_k v^3 \|^2_{L_h^{(\frac 38 -\frac 1{2p} +\frac{\delta}4
+\frac{\epsilon}4)^{-1}}} \Bigr\|_{L_v^1} \notag \\
\lesssim &
 \| |\nabla_h|^{-2\delta} \omega \|_{L_v^{\frac 1 {\epsilon}} L_h^{(\frac 14 +\frac 1p -\frac{\delta}2-\frac{\epsilon}2)^{-1}}}
  \cdot \| \partial_k v^3 \|^2_{L_v^{\frac 2 {1-\epsilon}} L_h^{(\frac 38 -\frac 1{2p} +\frac{\delta}4
+\frac{\epsilon}4)^{-1}}} \notag \\
\lesssim & \| |\nabla|^{1-\frac 2p} \omega\|_r
\cdot \| |\nabla|^{\frac 54+\frac 1p -\frac{\delta} 2} v^3 \|_2^2 \notag \\
\lesssim & \| \omega\|_r^{\frac 2p} \| \nabla \omega\|_r^{1-\frac 2p}
\cdot \| v^3 \|_{\dot H^{\frac 12+\frac 2p}} \cdot \| |\nabla_h|^{-\delta}
\partial^2 v^3 \|_2.
\end{align*}

Case 2b: $\int |\nabla_h|^{-1}\partial_3v^3 \partial_kv^3 \cdot \nabla_h |\nabla_h|^{-2\delta}
\partial_k v^3 dx$. We can use a similar commutator estimate as above to derive
\begin{align*}
 &
 \int |\nabla_h|^{-1}\partial_3v^3 \partial_kv^3 \cdot \nabla_h |\nabla_h|^{-2\delta}
\partial_k v^3 dx\notag \\
\lesssim &
\| |\nabla_h|^{-2\delta} \partial_3 v^3 \|_{L_v^2 L_h^{(\frac 38-\frac 34\delta)^{-1}}}
\cdot \| \partial_k v^3 \|^2_{L_v^4 L_h^{(\frac 5 {16} +\frac 38 \delta)^{-1}} } \notag \\
\lesssim &
\| |\nabla|^{\frac 54-\frac {\delta} 2} v^3 \|_2
\cdot \| |\nabla|^{\frac {13} 8 -\frac 3 4 \delta} v^3 \|_2^2 \notag \\
\lesssim & \| |\nabla|^{\frac 5 4 -\frac {\delta} 2} v^3 \|_2^2 \cdot \| |\nabla|^{2-\delta} v^3\|_2
\notag \\
\lesssim& \|  |\nabla_h|^{-\delta} \partial v^3  \|_2^{\frac 2p}
\cdot \| v^3 \|_{\dot H^{\frac 12+\frac 2p}} \cdot \| |\nabla_h|^{-\delta} \partial^2 v^3 \|_2^{2-\frac 2p}.
\end{align*}

\subsubsection{Estimate of \eqref{c5}}
We only need to deal with the expression for $l,m \in \{1,2\}$:
\begin{align}
 & \int \partial_l v^m \partial_m v^l \mathcal R_3 (|\nabla_h|^{-2\delta} \partial_k v^3) dx \notag \\
 = & \int \mathcal R_2( \partial_3 v^3) \cdot \mathcal R_2( \partial_3 v^3) \cdot \mathcal R_3( |\nabla_h|^{-2\delta}
 \partial_k v^3) dx \label{tmp_c5_a} \\
 & + \int \mathcal R_2( \omega) \cdot \mathcal R_2(\partial_3 v^3)
 \cdot \mathcal R_3( |\nabla_h|^{-2\delta} \partial_k v^3) dx \label{tmp_c5_b} \\
 & + \int  \mathcal R_2( \omega) \cdot \mathcal R_2(\omega)
 \cdot \mathcal R_3( |\nabla_h|^{-2\delta} \partial_k v^3) dx,\label{tmp_c5_c}  \\
 & \quad +\cdots. \notag
 \end{align}
where in the above $\mathcal R_2$, $\mathcal R_3$ denote Riesz type operators in
$x_h=(x_1,x_2)$ and the whole space $\mathbb R^3$ respectively.
The notation ``$\cdots$" denotes other omitted terms in the summation which can be represented
by either \eqref{tmp_c5_a}, \eqref{tmp_c5_b} or \eqref{tmp_c5_c}.
Clearly
\begin{align*}
|\eqref{tmp_c5_a}|
& \lesssim
\| |\nabla_h|^{-2\delta} \partial_k v^3 \|_{L_v^2 L_h^{(\frac 38-\frac 34\delta)^{-1}}}
\cdot \| \partial_3 v^3 \|^2_{L_v^4 L_h^{(\frac 5 {16} +\frac 38 \delta)^{-1}} } \notag \\
& \lesssim \|  |\nabla_h|^{-\delta} \partial v^3  \|_2^{\frac 2p}
\cdot \| v^3 \|_{\dot H^{\frac 12+\frac 2p}} \cdot \| |\nabla_h|^{-\delta} \partial^2 v^3 \|_2^{2-\frac 2p}.
\end{align*}
On the other hand,
\begin{align*}
|\eqref{tmp_c5_b}|
\lesssim &
 \|  \omega \|_{L_v^{\frac 1 {\epsilon}} L_h^{(\frac 14 +\frac 1p +\frac{\delta}2-\frac{\epsilon}2)^{-1}}}
  \cdot \| \partial v^3 \|_{L_v^{\frac 2 {1-\epsilon}} L_h^{(\frac 38 -\frac 1{2p} +\frac{\delta}4
+\frac{\epsilon}4)^{-1}}}
\cdot \| |\nabla_h|^{-2\delta}  \partial v^3 \|_{L_v^{\frac 2 {1-\epsilon}} L_h^{(\frac 38 -\frac 1{2p} -\frac{3\delta}4
+\frac{\epsilon}4)^{-1}}}
 \notag \\
\lesssim & \| |\nabla|^{1-\frac 2p} \omega\|_r
\cdot \| |\nabla|^{\frac 54+\frac 1p -\frac{\delta} 2} v^3 \|_2^2 \notag \\
\lesssim & \| \omega\|_r^{\frac 2p} \| \nabla \omega\|_r^{1-\frac 2p}
\cdot \| v^3 \|_{\dot H^{\frac 12+\frac 2p}} \cdot \| |\nabla_h|^{-\delta}
\partial^2 v^3 \|_2.
\end{align*}
Finally for \eqref{tmp_c5_c} we can integrate by parts in $\partial_k$. Then
\begin{align*}
| \eqref{tmp_c5_c}|
\lesssim&
\| \nabla \omega \|_r
\cdot \| \omega \|_{L_v^{\frac 1 {\epsilon}} L_h^{(\frac 14 +\frac 1p +\frac {\delta}2
-\frac {\epsilon} 2)^{-1}}}
\cdot \| |\nabla_h|^{-2\delta} v^3
\|_{L_v^{(1-\frac 1r -\epsilon)^{-1}}
L_h^{(\frac 14-\frac 5 6 \delta -\frac 1p +\frac{\epsilon} 2)^{-1} } } \notag \\
\lesssim & \| \nabla \omega \|_r
\cdot \| |\nabla|^{1-\frac 2p} \omega \|_r
\cdot \| v^3 \|_{\dot H^{\frac 12 +\frac 2p}}.
\end{align*}

\section{Gronwall and proof of main theorem}
\subsection{Gronwall for $p=2$}
The estimate of $\omega$ is
\begin{align}\label{d1}
 \begin{split}
   &\frac{1}{r}\frac{d}{dt} \|\omega\|_{L^r}^r
    +(r-1)\|\nabla |\omega|^{\frac r2} \|_{L^2}^2 \\
\lesssim & \|v^3\|_{\dot{H}^{\frac32}}\|\omega\|_{L^{r}}^{r-1}
   \|\nabla \omega\|_{L^{r}}
   + \|v^3\|_{\dot{H}^{\frac32}}\||\nabla_h|^{-\delta}\partial_3^2v^3\|_{L^2}
           \|\omega\|_{L^r}^{r-1}\\
      &
  + (\||\nabla_h|^{-\delta}\partial^2 v^3\|_{L^2}^{\frac12+\delta+
       \epsilon_1}
       +\||\nabla_h|^{-\delta}\partial^2 v^3\|_{L^2}^{\frac12+\delta-
       \epsilon_1})\cdot
       \frac{1}{1+\|\omega\|^{100}_{L^r}}
       \| \nabla |\omega|^{\frac r2} \|_{L^2}\\
 &+\sqrt{\log(10+\||\nabla_h|^{-\delta}\partial v^3\|_{L^2}
  +\|\omega\|_{L^r})}\big(\|v^3\|_{\dot{H}^{\frac{3}{2}}}
  +1\big)\|\nabla \omega\|_{L^{r}} \|\omega\|_{L^r}^{r-1}.
  \end{split}
\end{align}
The estimate of $v^3$ is
\begin{align}\label{d2}
 \begin{split}
 &\frac12 \frac{d}{dt}\bigg(\sum_{k=1}^3\||\nabla_h|^{-\delta}\partial_k v^3\|^2_{L^2}\bigg)
    +\sum_{k=1}^3\||\nabla_h|^{-\delta}\partial_k \nabla v^3\|^2_{L^2}\\
\lesssim &  \|v^3\|_{\dot{H}^{\frac32}}\||\nabla_h|^{-\delta}\partial^2 v^3\|_{L^2}
        \||\nabla_h|^{-\delta}\partial_k v^3\|_{L^2}
   +\|\nabla \omega\|_{L^{r}}
    \||\nabla_h|^{-\delta}\partial v^3\|_{L^2} \|v^3\|_{\dot{H}^{\frac32}}\\
  &+\|\omega\|_{L^r}\|v^3\|_{\dot{H}^{\frac32}}
    \||\nabla_h|^{-\delta}\partial^2 v^3\|_{L^2}
   +\|\omega\|_{L^r} \|\nabla \omega\|_{L^r}
    \|v^3\|_{\dot{H}^{\frac32}}.
  \end{split}
\end{align}
Multiplying inequality \eqref{d1} with $\|\omega\|_{L^r}^{2-r}$ yields
\begin{align}\label{d3}
 \begin{split}
   &\frac{1}{2}\frac{d}{dt}\|\omega\|_{L^r}^{2}
    +(r-1)\|\omega\|_{L^r}^{2-r}
    \|\nabla \omega|^{\frac r2} \|_{L^2}^2\\
\lesssim & \|v^3\|_{\dot{H}^{\frac32}}\|\omega\|_{L^{r}}
   \|\nabla \omega\|_{L^{r}}
   + \|v^3\|_{\dot{H}^{\frac32}}\||\nabla_h|^{-\delta}\partial_3^2v^3\|_{L^2}
           \|\omega\|_{L^r}\\
      &
  +(\||\nabla_h|^{-\delta}\partial^2 v^3\|_{L^2}^{\frac12+\delta+
       \epsilon_1}+\||\nabla_h|^{-\delta}\partial^2 v^3\|_{L^2}^{\frac12+\delta-
       \epsilon_1})\cdot
       \frac{1}{1+\|\omega\|^{100}_{L^r}}
    \|\nabla |\omega|^{\frac r2} \|_{L^2}
    \|\omega\|_{L^r}^{2-r}\\
 &+\sqrt{\log(10+\||\nabla_h|^{-\delta}\partial v^3\|_{L^2}
  +\|\omega\|_{L^r})}\big(\|v^3\|_{\dot{H}^{\frac{3}{2}}}
  +1\big)\|\nabla \omega\|_{L^{r}} \|\omega\|_{L^r}.
  \end{split}
\end{align}
Then, it follows from \eqref{d3} that
\begin{align}\label{d4}
 \begin{split}
   &\frac{d}{dt}\|\omega\|_{L^r}^{2}
    +\|\omega\|_{L^r}^{2-r}
    \| \nabla  |\omega|^{\frac r2} \|_{L^2}^2 \\
\le &\;  C\|v^3\|_{\dot{H}^{\frac32}}^2\|\omega\|_{L^{r}}^{2}
         + \frac 1 {100} \||\nabla_h|^{-\delta}\partial^2v^3\|^2_{L^2}\\
      &+C \log(10+\||\nabla_h|^{-\delta}\partial v^3\|_{L^2}
      +\|\omega\|_{L^r})\big(\|v^3\|^2_{\dot{H}^{\frac{3}{2}}}+1\big)
      \|\omega\|_{L^{r}}^{2}.
  \end{split}
\end{align}
In addition, we know from \eqref{d2} that
\begin{align}\label{d5}
 \begin{split}
   &\frac{d}{dt}\||\nabla_h|^{-\delta}\partial v^3\|^2_{L^2}
    +\||\nabla_h|^{-\delta}\partial^2 v^3\|^2_{L^2}\\
\le & \; C  \|v^3\|^2_{\dot{H}^{\frac32}}\||\nabla_h|^{-\delta}\partial v^3\|^2_{L^2}
   + C \|v^3\|_{\dot{H}^{\frac32}}^2\|\omega\|^{2}_{L^r}
  +\frac 1 {100} \|\omega\|_{L^r}^{2-r}
    \||\nabla |\omega|^{\frac r2} \|_{L^2}^2.
  \end{split}
\end{align}
Adding \eqref{d4} and \eqref{d5} together, one has
\begin{align}\label{d6}
 \begin{split}
   &\frac{d}{dt}\big(\|\omega\|_{L^r}^{2}
   +\||\nabla_h|^{-\delta}\partial v^3\|^2_{L^2}\big)
   +\frac 12 \|\omega\|_{L^r}^{2-r}
    \| \nabla |\omega|^{\frac r2} \|_{L^2}^2
    +\frac 12 \||\nabla_h|^{-\delta}\partial^2 v^3\|^2_{L^2}\\
\le &\; C \|v^3\|_{\dot{H}^{\frac32}}^2\big(\|\omega\|_{L^r}^{2}
   +\||\nabla_h|^{-\delta}\partial v^3\|^2_{L^2}\big)
   \\
      &+C\log\big(10+\||\nabla_h|^{-\delta}\partial v^3\|_{L^2}
      +\|\omega\|_{L^r}\big)\big(\|v^3\|^2_{\dot{H}^{\frac{3}{2}}}+1\big)
      \|\omega\|_{L^r}^{2}.
  \end{split}
\end{align}
Using Gronwall inequality, we obtain that
\begin{align}\label{d7}
 \begin{split}
   &\|\omega(t)\|_{L^r}^{2}
   +\||\nabla_h|^{-\delta}\partial v^3(t)\|^2_{L^2}\\
\leq &\big(\|\omega_0\|_{L^r}^{2}
   +\||\nabla_h|^{-\delta}\partial u_0^3\|^2_{L^2}+10\big)
   ^{\exp\{C\int_0^t(\|v^3(s)\|^2_{\dot{H}^{\frac32}}
     +1)ds\}}.
  \end{split}
\end{align}
It also follows from \eqref{d6} and \eqref{d7} that
\begin{align*}
   &\int_0^t\|\omega\|_{L^r}^{2-r}
    \big\||\nabla \omega| |\omega|^{\frac r2-1}\big\|_{L^2}^2ds
    +\int_0^t\||\nabla_h|^{-\delta}\partial^2 v^3\|^2_{L^2}ds\\
\leq & C\big(\|\omega\|_{L_t^\infty L_x^r}^{2}
   +\||\nabla_h|^{-\delta}\partial v^3\|^2_{L_t^\infty L_x^2}\big)
   \log\big(10+\||\nabla_h|^{-\delta}\partial v^3\|_{L_t^\infty L_x^2}
      +\|\omega\|_{L_t^\infty L_t^r}\big)\\
      &\big[\int_0^t(\|v^3(s)\|^2_{\dot{H}^{\frac32}}
     +1)ds\big].
\end{align*}

\subsection{Gronwall for $2<p<\infty$}
Now we consider the case when $2<p<\infty$. The estimate for $\omega$ is
\begin{align}\label{d8}
 \begin{split}
&\frac{1}{r}\frac{d}{dt}\||\omega|^{\frac r2}\|_{L^2}^2
    +\frac{4(r-1)}{r^2}\|\nabla |\omega|^{\frac r2}\|_{L^2}^2\\
\leq & C\|v^3\|_{\dot{H}^{\frac12 + \frac{2}{p}}}
        \||\nabla_h|^{-\delta}\partial_3^2 v^3\|_{L^2}
       \| |\omega|^{\frac{r}{2}} \|^{1-\frac2r+\frac{2}{p}}_{L^2}
       \| \nabla |\omega|^{\frac r2} \|^{1-\frac{2}{p}}_{L^2}\\
   & +C\|v^3\|_{\dot{H}^{\frac12+\frac{2}{p}}}
  \||\omega|^{\frac r2}\|_{L^2}^{\frac2p}
  \|\nabla |\omega|^{\frac r2}\|_{L^2}^{2-\frac2p}.
  \end{split}
\end{align}

Multiplying inequality \eqref{d8} with $\||\omega|^{\frac r2}\|_{L^2}^{\frac4r-2}$, we get
\begin{align}\label{d9}
\begin{split}
   &\frac12\frac{d}{dt}\||\omega|^{\frac r2}\|_{L^2}^{\frac4r}
    +\frac{4(r-1)}{r^2}\||\omega|^{\frac r2}\|_{L^2}^{\frac4r-2}\|\nabla |\omega|^{\frac r2}\|_{L^2}^2\\
\leq &C\|v^3\|_{\dot{H}^{\frac12+\frac{2}{p}}}
    \||\omega|^{\frac r2}\|_{L^2}^{\frac{4}{rp}}
    \big(\||\omega|^{\frac r2}\|_{L^2}^{\frac2r-1}\|\nabla |\omega|^{\frac r2}\|_{L^2}\big)^{1-\frac2p}
    \||\nabla_h|^{-\delta}\partial_3^2 v^3\|_{L^2}\\
   & +C\|v^3\|_{\dot{H}^{\frac12+\frac{2}{p}}}
  \||\omega|^{\frac r2}\|_{L^2}^\frac{4}{rp}
  \big(\||\omega|^{\frac r2}\|_{L^2}^{\frac2r-1}\|\nabla |\omega|^{\frac r2}\|_{L^2}\big)^{2-\frac2p}\\
\leq &C\|v^3\|_{\dot{H}^{\frac12+\frac{2}{p}}}^p \||\omega|^{\frac r2}\|_{L^2}^{\frac{4}{r}}
    +\frac 1 {100} \||\nabla_h|^{-\delta}\partial_3^2 v^3\|_{L^2}^2.
  \end{split}
\end{align}

The estimate for $\partial_k v^3$ is
\begin{align}\label{d10}
\begin{split}
&\frac12 \frac{d}{dt}\bigg(\sum_{k=1}^3\||\nabla_h|^{-\delta}\partial_k v^3\|^2_{L^2}\bigg)
    +\sum_{k=1}^3\||\nabla_h|^{-\delta}\partial_k \nabla v^3\|^2_{L^2}\\
\leq  &C\|v^3\|_{\dot{H}^{\frac12+\frac2p}}
    \||\nabla_h|^{-\delta}\partial_k v^3\|^{\frac2p}_{L^2}
    \||\nabla_h|^{-\delta}\partial_k \nabla v^3\|^{2-\frac2p}_{L^2}\\
    &+ C\|v^3\|_{\dot{H}^{\frac12+\frac2p}}
    \||\omega|^{\frac r2}\|_{L^2}^\frac{4}{rp}
  \big(\||\omega|^{\frac r2}\|_{L^2}^{\frac2r-1}\|\nabla |\omega|^{\frac r2}\|_{L^2}\big)^{2-\frac2p}\\
   &+C\|v^3\|_{\dot{H}^{\frac12+\frac2p}}
   \||\omega|^{\frac r2}\|_{L^2}^{\frac{4}{rp}}
    \big(\||\omega|^{\frac r2}\|_{L^2}^{\frac2r-1}\|\nabla |\omega|^{\frac r2}\|_{L^2}\big)^{1-\frac2p}
    \||\nabla_h|^{-\delta}\partial_k \nabla v^3\|_{L^2} \\
    & \; + C \| v^3 \|_{\dot H^{\frac 12 +\frac 2p}} \cdot \| \nabla \omega \|_r
 \cdot \| |\nabla_h|^{-\delta} \partial v^3 \|_2^{\frac 2p}
 \cdot \| |\nabla_h|^{-\delta} \partial^2 v^3 \|_2^{1-\frac 2p}.
    \end{split}
\end{align}
Adding \eqref{d9} and \eqref{d10} together and using Young inequality, one has
\begin{align*}
   &\frac{d}{dt}\bigg(\||\omega|^{\frac r2}\|_{L^2}^{\frac4r}
   + \sum_{k=1}^3\||\nabla_h|^{-\delta}\partial_k v^3\|_{L^2}^2\bigg)\\
    &\quad+\bigg(\||\omega|^{\frac r2}\|_{L^2}^{\frac4r-2}\|\nabla |\omega|^{\frac r2}\|_{L^2}^2
    +\sum_{k=1}^3\||\nabla_h|^{-\delta}\nabla \partial_k v^3\|_{L^2}^2\bigg)\\
&\leq  C\|v^3\|^p_{\dot{H}^{\frac12+\frac{2}{p}}}
 \bigg(\||\omega|^{\frac r2}\|_{L^2}^{\frac4r}
   +\sum_{k=1}^3\||\nabla_h|^{-\delta}\partial_k v^3\|_{L^2}^2\bigg).
\end{align*}
Therefore, standard Gronwall inequality shows that
\begin{align}\label{d11}
\begin{split}
   &\||\omega|^{\frac r2}(t)\|_{L^2}^{\frac4r}
   + \sum_{k=1}^3\||\nabla_h|^{-\delta}\partial_k v^3(t)\|_{L^2}^2\\
    &\quad+\int_0^t\||\omega|^{\frac r2}(s)\|_{L^2}^{\frac4r-2}\|\nabla |\omega|^{\frac r2}(s)\|_{L^2}^2ds
    +\sum_{k=1}^3\int_0^t\||\nabla_h|^{-\delta}\nabla \partial_k v^3(s)\|_{L^2}^2ds\\
&\leq
\bigg( \||\omega_0|^{\frac r2}\|_{L^2}^{\frac4r}
   +\sum_{k=1}^3\||\nabla_h|^{-\delta}\partial_k v_0^3\|_{L^2}^2\bigg)
   \exp\{C\int_0^t \|v^3(s)\|^p_{\dot{H}^{\frac12+\frac{2}{p}}}ds\}.
\end{split}
\end{align}

Now we are ready to prove the main theorem.

\begin{proof}[Proof of Theorem \ref{t1.2} (for $2\leq p<\infty$)]
By smoothing estimates we may assume without loss of generality
that $v_0 \in \dot H^{\frac 12} \cap \dot H^1$, and $\Omega_0,\, \nabla^4 \Omega_0 \in L^{r_0}$.
With these assumptions (and propagation of regularity) we note that the auxiliary norms
$\|\omega\|_r$, $\| |\nabla_h|^{-\delta} \partial v^3 \|_2$ are well defined for any $r \in (2-\epsilon_0,2]$ with $\epsilon_0>0$ sufficiently small, during the life span of the local solution.

Now to control the local solution,
by using Proposition \ref{vH1} and Remark \ref{rem3.3}, it suffices for us to control
$\|\omega \|_{L_t^p \dot H^{-\frac 12 +\frac 2p} (0,T^\ast)}$ if $2\le p\le 4$ and
$\|\omega \|_{L_t^{\infty} L_x^{\tilde r}}$,  for some $\tilde r$ satisfying $\frac 12
<\frac 1 {\tilde r} <\frac 23 (1-\frac 1p)$, if $4<p<\infty$.
Consider first the case $4<p<\infty$. We shall take $r$ sufficiently close to $2$.
By the Gronwall estimates derived in previous sections, we have uniform estimates on
$\|\omega\|_r$.  It follows easily that the solution
remains regular.

Next for $2\le p\le 4$ we can  take $r$ sufficiently close to $2$ satisfying also
$\frac 2p + \frac 3r -2 >0$. Then
\begin{align*}
\| \omega \|_{\dot H^{-\frac 12+\frac 2p}}
& \lesssim \| |\nabla|^{\frac 2p+\frac 3r-2} \omega \|_r \notag \\
& \lesssim \| \omega \|_r^{3-\frac 3r -\frac 2p} \| \nabla \omega \|_r^{\frac 2p+\frac 3r-2}
\notag \\
& \lesssim  \| \nabla ( |\omega|^{\frac r2} ) \|_2^{\frac 2p+\frac 3r -2}
\cdot \| \omega \|_r^{r-\frac rp -\frac 12}
\end{align*}
Noting that $0<\frac 2p+\frac 3r -2<\frac 2p$ and $\|\nabla(|\omega|^{\frac r2} )\|_{L_t^2L_x^2}
\lesssim 1$,   it follows easily that
$$
\| \omega \|_{L_t^p \dot H^{-\frac 12+\frac 2p} (0,T^\ast)} <\infty.
$$
Thus the solution remains regular.
\end{proof}

\section*{Acknowledgements}
The first author was  in part
supported by NSFC (No. 11626075) and Zhejiang Province Science fund for Youths (No. LQ17A010007). D. Li was supported in part by an Nserc Discovery grant.
Z. Lei and N. Zhao was in part supported by NSFC (grant No. 11421061), National Support Program for Young Top-Notch Talents and SGST 09DZ2272900.


\end{document}